\documentclass[11pt]{article}

\usepackage[a4paper,margin=2.8cm]{geometry}
\usepackage{amsmath,amssymb,amsthm,mathtools,esint,mathrsfs}
\usepackage{microtype}
\usepackage{enumitem}
\usepackage{aliascnt}
\usepackage{hyperref}

\hypersetup{
  colorlinks=true,
  linkcolor=magenta,
  citecolor=magenta,
  urlcolor=magenta
}

\newtheorem{theorem}{Theorem}[section]
\newaliascnt{proposition}{theorem}
\newtheorem{proposition}[proposition]{Proposition}
\aliascntresetthe{proposition}
\newaliascnt{lemma}{theorem}
\newtheorem{lemma}[lemma]{Lemma}
\aliascntresetthe{lemma}
\newaliascnt{corollary}{theorem}
\newtheorem{corollary}[corollary]{Corollary}
\aliascntresetthe{corollary}
\theoremstyle{definition}
\newaliascnt{definition}{theorem}
\newtheorem{definition}[definition]{Definition}
\aliascntresetthe{definition}
\newaliascnt{remark}{theorem}
\newtheorem{remark}[remark]{Remark}
\aliascntresetthe{remark}

\numberwithin{equation}{section}

\newcommand{\R}{\mathbb R}

\newcommand{\Hh}{\mathcal H}
\newcommand{\Om}{\Omega}

\newcommand{\one}{\mathbf 1}

\allowdisplaybreaks[2]

\newcommand{\Addresses}{{ 
 \bigskip
 \textit{E-mail address}, C.~Nitsch: \texttt{c.nitsch@unina.it}

  \medskip 
 
 \noindent\textit{E-mail address}, C.~Trombetti: \texttt{cristina@unina.it} 
   \medskip

 \noindent\textsc{Dipartimento di Matematica e Applicazioni ``R. Caccioppoli'', Universit\`a degli studi di Napoli Federico II, Via Cintia, Complesso Universitario Monte S. Angelo, 80126 Napoli, Italy.}
 
 \medskip

 \noindent \textit{E-mail address}, F.~Villone: \texttt{f.villone@ssmeridionale.it} 
  
   \medskip 
 
  \noindent\textsc{Mathematical and Physical Sciences for Advanced Materials and Technologies, Scuola Superiore Meridionale, Largo San Marcellino 10, 80138 Napoli, Italy. }

 \par\nopagebreak 

}}

\title{A Quantitative Robin Faber–Krahn Inequality in Terms of Eigenfunction Distance}
\author{Carlo Nitsch, Cristina Trombetti, Federico Villone}
\date{}

\begin{document}

\maketitle

\begin{abstract}
 We establish a quantitative stability estimate for the first eigenvalue of the Laplacian with positive Robin boundary conditions. For every open set \(\Omega\subset\mathbb R^n\) of prescribed volume, we prove that
\[
\lambda_{1,\beta}(\Omega)-\lambda_{1,\beta}(B)
\geq c\|u_\Omega-u_B\|_{L^2(\mathbb R^n)}^4,
\]
where \(B\) is a ball achieving the Fraenkel asymmetry of \(\Omega\), and \(u_\Omega\) and \(u_B\) are nonnegative, \(L^2\)-normalized first Robin eigenfunctions, extended by zero outside their respective domains. The constant \(c>0\) depends only on the dimension, the Robin parameter, and the prescribed volume. The exponent \(4\) is optimal.

\end{abstract}
\begin{center}
\begin{minipage}{10cm}
\small
\tableofcontents
\end{minipage}
\end{center}
\section{Introduction}\label{sec:introduction}

Let $n\ge 2$ and let $\Omega\subset\R^n$ be a bounded open set with smooth boundary.
Let $\nu_\Omega$ denote the unit outer normal to $\Omega$. Fixing $\beta>0$, we consider
the following eigenvalue problem with Robin boundary conditions:
\begin{equation}\label{eq:intro-robin-strong}
\begin{cases}
-\Delta u_\Omega = \lambda_{1,\beta}(\Omega)u_\Omega& \text{in }\Omega,\\[5pt]
\partial_{\nu_\Omega}u_\Omega+\beta u_\Omega=0
& \text{on }\partial\Omega.
\end{cases}
\end{equation}
The eigenvalue problem \eqref{eq:intro-robin-strong} admits a discrete spectrum,
and its first eigenvalue $\lambda_{1,\beta}(\Omega)$ can be variationally
characterized as
\begin{equation}\label{eq:intro-robin-quotient}
\lambda_{1,\beta}(\Omega)
=
\min_{v\in H^1(\Omega)\setminus\{0\}}
\frac{
\displaystyle\int_\Omega |\nabla v|^2\,dx
+\beta\int_{\partial\Omega}v^2\,d\Hh^{n-1}
}{
\displaystyle\int_\Omega v^2\,dx
}.
\end{equation}
Moreover, the first eigenvalue is simple when $\Omega$ is connected, and an associated eigenfunction
 satisfies
\begin{equation}\label{eq:intro-robin-weak}
\int_\Omega \nabla u_\Omega \cdot\nabla\psi\,dx
+\beta\int_{\partial\Omega}u_\Omega \psi\,d\Hh^{n-1}
=
\lambda_{1,\beta}(\Omega)\int_\Omega u_\Omega\psi\,dx
\end{equation}
for every $\psi\in H^1(\Omega)$.

Throughout the paper, first eigenfunctions are chosen nonnegative and
normalized in $L^2(\Omega)$, namely
\[
\int_\Omega u_\Omega^2\,dx=1.
\]

A natural shape optimization problem for the first eigenvalue is to determine, among all domains of prescribed volume, which shape minimizes the principal eigenvalue. In the Dirichlet case, this question is answered by the classical Faber--Krahn inequality \cite{Faber1923,Krahn1925}, which states that the balls minimize the first eigenvalue under a volume constraint. The corresponding result for the Robin problem is the Bossel--Daners inequality. For positive Robin parameter, the planar case for smooth domains was proved by Bossel in \cite{Bossel1986,Bossel1988}, while Daners established the result for bounded Lipschitz domains in arbitrary dimension in \cite{Daners2006}. More precisely, if $B$ is a ball such that $|B|=|\Omega|$, then
\begin{equation}\label{eq:intro-bossel-daners}
    \lambda_{1,\beta}(\Omega)
    \geq
    \lambda_{1,\beta}(B).
\end{equation}
 For sufficiently smooth domains, equality holds only for balls, up to translations \cite{DanersKennedy2007}. 

For our purposes, it will be necessary to work with the first Robin eigenvalue on arbitrary open sets of finite measure. We therefore adopt the $SBV^{1/2}$ relaxation introduced in \cite{BucurGiacomini2010,BucurGiacomini2015}, which coincides with the variational characterization \eqref{eq:intro-robin-quotient} whenever $\Omega$ is a bounded Lipschitz set. The precise definition and the corresponding variational framework are recalled in \autoref{sec:setting}. Variational proofs of \eqref{eq:intro-bossel-daners} under these weaker regularity assumptions were subsequently obtained in \cite{BucurDaners2010,BucurGiacomini2010,BucurGiacomini2015}.

Since equality in \eqref{eq:intro-bossel-daners} is attained only by balls, up to translations, it is natural to ask for a quantitative version of this rigidity statement. Thus, one would like to understand whether an almost optimal domain, in the sense that its first Robin eigenvalue is close to that of the ball with the same volume, must itself be close to a ball. In other words, the goal of a quantitative stability estimate is to control a suitable  distance from $\Omega$ to the family of balls in terms of the eigenvalue deficit

$$
\lambda_{1,\beta}(\Omega)-\lambda_{1,\beta}(B_R),
\qquad |B_R|=|\Omega|.
$$
A natural way to measure this geometric distance is through
the Fraenkel asymmetry, recalled below. 

\begin{definition}\label{def:fraenkel-ball}
Let $\Omega\subset\mathbb R^n$ be a measurable set with
$0<|\Omega|<\infty$, and let $B_R$ be a ball such that
$|B_R|=|\Omega|$.
The Fraenkel asymmetry of $\Omega$ is defined by
\begin{equation}\label{eq:intro-fraenkel}
\mathcal A(\Omega)
:=\inf_{x\in\mathbb R^n}
\frac{|\Omega\triangle B_R(x)|}{|\Omega|}.
\end{equation}
We say that a ball $B$ achieves the Fraenkel asymmetry
of $\Omega$ if
\[
|B|=|\Omega|,
\qquad
\mathcal A(\Omega)
=\frac{|\Omega\triangle B|}{|\Omega|}.
\]
Such a ball always exists, although it may not be unique.
\end{definition}

In the Dirichlet setting, quantitative stability estimates developed from the first bounds in \cite{Melas1992,FMP2009} to the sharp inequality \cite{BDV2015}

$$
 \lambda_1^D(\Omega)-\lambda_1^D(B_R)
 \ge c(n,V)\,\mathcal A(\Omega)^2.
$$

The quadratic exponent is optimal. Its proof uses the selection principle introduced in \cite{CicaleseLeonardi2012}, and the Kohler--Jobin relation with torsional rigidity \cite{KohlerJobin1978,KohlerJobin1982}. See also \cite{BrascoDePhilippis2017} for a survey on quantitative spectral inequalities.

In the Robin setting, the sharp quantitative counterpart of the Bossel--Daners inequality was established in \cite{BucurFeroneNitschTrombetti2018}. In the relaxed framework introduced above, it states that
\begin{equation}\label{eq:intro-qfk}
\delta_\beta(\Omega)
:=\lambda_{1,\beta}(\Omega)-\lambda_{1,\beta}(B_R)
\ge c(n,\beta,V)\,\mathcal A(\Omega)^2,
\qquad |\Omega|=|B_R|=V.
\end{equation}
As in the Dirichlet case, the exponent $2$ is sharp. The proof, however, follows a substantially different route. In particular, no Robin analogue of the Kohler--Jobin inequality is currently known, so the argument cannot be reduced to a quantitative estimate for an auxiliary torsional functional.  Related methods for Robin type problems were developed in \cite{BGN2021}.

Another possible approach to quantitative spectral inequalities is based on Talenti's principle \cite{Talenti1976,AlvinoNitschTrombetti2023} and its quantitative counterpart \cite{AmatoBarbatoMasielloPaoli2024,AmatoBarbatoCitoMasielloPaoli2026,AL26}. In particular, in dimension two, the authors of \cite{AmatoBarbatoCitoMasielloPaoli2026} provide an alternative proof of the sharp quantitative Bossel–Daners inequality \eqref{eq:intro-qfk}.



Although the Fraenkel asymmetry is a natural geometric quantity in the
quantitative Bossel--Daners inequality, it does not capture all the
information that is relevant for the Robin problem. Indeed, it only measures
the distance between sets in measure, whereas the Robin eigenvalue is also
sensitive to the creation of new portions of boundary. This discrepancy can
already be seen within the class of Lipschitz domains. For instance, starting
from a ball $B$, one may remove a thin slab crossing the ball, and consider
\[
    \Omega_\varepsilon
    =
    B\setminus\{ |x_1|<\varepsilon\}.
\]
 Then
$\mathcal A(\Omega_\varepsilon)\to0$ as $\varepsilon\to0$, since the removed
region has vanishing volume. On the other hand, in the limit the spectral problem
does not reduce to the Robin problem on the original ball. Thus, closeness in
Fraenkel asymmetry alone does not necessarily reflect closeness of the
corresponding eigenvalues.

This suggests that the Fraenkel asymmetry can be complemented by an
additional notion of distance, capable of detecting perturbations that are
small in measure but remain visible to the underlying PDE. 

The same phenomenon occurs in the Dirichlet setting, where modifications on
sets of zero measure but positive capacity may change the spectrum while
remaining invisible to the Fraenkel asymmetry. This idea is developed in
\cite{AKN2023ACF,AKN2023Linear}, where it is proved that, for a suitable
equal-volume ball $B$ and suitably normalized first eigenfunctions extended
by zero to $\mathbb R^n$,
\begin{equation}\label{eq:intro-akn}
 \lambda_1^D(\Omega)-\lambda_1^D(B)
 \geq
 c\left(
 |\Omega\triangle B|^2
 +
 \|u_\Omega-u_B\|_{L^2(\mathbb R^n)}^2
 \right).
\end{equation}
Hence the spectral deficit controls not only the distance of the domain from
a ball in measure, but also the distance between the corresponding first
eigenfunctions, which detects information that is not contained in the
Fraenkel asymmetry.

This point of view is further developed in \cite{AKN}, where the comparison
of first eigenfunctions is strengthened by controlling the difference between
the corresponding Dirichlet resolvents.

In this paper, we investigate the
analogous question for the first Robin eigenvalue, asking whether the Bossel--Daners deficit controls, in addition to the Fraenkel asymmetry, a suitable
distance between the corresponding first eigenfunctions.

More precisely, we have the following result.
\begin{theorem}\label{thm:main}
Let $n\geq 2$, $\beta>0$, and $V>0$.
There exists a constant $C=C(n,\beta,V)>0$
such that the following holds for every open set
$\Omega\subset\mathbb R^n$ with $|\Omega|=V$.

Let $B=B_R(x_\Omega)$ be any ball achieving the
Fraenkel asymmetry of $\Omega$, in the sense of \autoref{def:fraenkel-ball}.
Then, for every nonnegative, $L^2$-normalized
first Robin eigenfunction $u_\Omega$ on $\Omega$,
\begin{equation}
\label{eq:main-estimate}
\|u_\Omega-u_B\|_{L^2(\mathbb R^n)}
\leq C\,\delta_\beta(\Omega)^{1/4},
\end{equation}
where $u_B$ is the nonnegative, $L^2$-normalized
first Robin eigenfunction of $B$.
Both eigenfunctions are extended by zero outside
their respective domains.

The exponent $1/4$ is optimal, even within
the class of smooth connected domains.
\end{theorem}

We point out that \autoref{thm:main} already includes the geometric asymmetry. Indeed, there exists a constant $m_B$ depending only on $n, V$ and $\beta$ such that 
(we defer to \eqref{eq:phi-bounds} for the details), the first Robin eigenfunction of the ball satisfies
\(u_B\ge m_B>0\) in \(B\). Since \(u_\Omega\) vanishes almost everywhere outside
\(\Omega\) and \(|\Omega|=|B|=V\), we have
\begin{equation}\label{eq:intro-asymmetry-control}
 \|u_\Omega-u_B\|_{L^2(\mathbb R^n)}^2
 \ge \int_{B\setminus\Omega}u_B^2\,dx
 \ge m_B^2|B\setminus\Omega|
 =\frac{m_B^2V}{2}\mathcal A(\Omega).
\end{equation}
Consequently, after changing the constant, \autoref{thm:main} can
equivalently be written as
\begin{equation}\label{eq:main-combined}
 \delta_\beta(\Omega)
 \ge c(n,\beta,V)\left(
 \mathcal A(\Omega)^2
 +\|u_\Omega-u_B\|_{L^2(\mathbb R^n)}^4
 \right).
\end{equation}
 We stress that our proof uses
\eqref{eq:intro-qfk}; this is therefore a strengthening of that estimate,
rather than an independent proof of it.

It is worth emphasizing that the proof of \autoref{thm:main} cannot be obtained by a direct adaptation of the Dirichlet argument in \cite{AKN2023ACF}. A first substantial obstruction is that the Dirichlet proof relies on the Kohler--Jobin inequality, which provides a crucial link between the spectral deficit and the torsional rigidity. No analogous inequality suitable for this purpose is currently known for the Robin problem, so this part of the argument has no immediate counterpart in the present setting.

There is also a second, more structural difficulty. The Dirichlet approach makes essential use of regularity properties and of a local analysis around the ball based on perturbed free boundary problems. These arguments are closely tied to the Dirichlet structure and do not transfer directly to the Robin case, where the boundary contribution enters the energy explicitly. In particular,  the regularity theory becomes substantially more delicate.

For these reasons, the proof of \autoref{thm:main} follows a different route, tailored to the Robin problem, rather than reproducing the Dirichlet strategy.

\medskip

Moreover, we establish a stronger stability estimate for smooth
perturbations of the ball. To state this result, we introduce the
class of nearly spherical domains.

Let $B=B_R(0)$ be the ball of volume $V$. For every
$\xi\in C^0(\partial B)$ with
$\|\xi\|_{L^\infty(\partial B)}<R/2$, we denote by $B^\xi$
the bounded domain whose boundary is given by
\[
\partial B^\xi
=
\{x+\xi(x)\nu_B(x):x\in\partial B\},
\]
where $\nu_B$ denotes the outward unit normal to $\partial B$.
Following \cite{DL} we define the class
of nearly spherical domains as follows. 
\begin{definition}\label{def:nearly-spherical}
Let $B=B_R(0)$ be the ball of volume $V$.
For $\eta\in(0,R/2)$ and $\gamma\in(0,1)$,
we define the class of nearly spherical domains by
\begin{equation}\label{eq:nearly-spherical}
\mathcal S_{\eta,\gamma}(B)
:=
\left\{
B^\xi:
\begin{array}{l}
\xi\in C^{2,\gamma}(\partial B),\quad 
\|\xi\|_{C^{2,\gamma}(\partial B)}\leq\eta,\\
|B^\xi|=|B|,\quad
\displaystyle\int_{B^\xi}x\,dx=0
\end{array}
\right\}.
\end{equation}
\end{definition}

For nearly spherical domains, we can compare the eigenfunctions
by identifying the perturbed domain with the reference ball.
More precisely, let $h_\xi$ be the harmonic extension of
$\xi/R$ to $B$ and define
\[
\Psi_\xi(x):=(1+h_\xi(x))x.
\]
We denote by $u_\xi$ and $u_B$ the positive, $L^2$-normalized
first Robin eigenfunctions of $B^\xi$ and $B$, respectively.
We prove the following result, whose proof is given in
\autoref{sec:nearly-spherical}.

\begin{theorem}\label{thm:nearly-spherical}
Let $n\geq 2$, $\beta,V>0$, and $\gamma\in(0,1)$.
Let $B=B_R(0)$ be the ball of volume $V$.

There exist constants
$\eta_0=\eta_0(n,\beta,V,\gamma)>0$ and
$C=C(n,\beta,V,\gamma)>0$ such that,
for every $B^\xi\in\mathcal S_{\eta_0,\gamma}(B)$,
the map $\Psi_\xi$ is a diffeomorphism from
$B$ onto $B^\xi$, and
\begin{equation}\label{eq:local-stability}
\begin{aligned}
&\|u_\xi\circ\Psi_\xi-u_B\|_{H^1(B)}^2\quad+
\|u_\xi-u_B\|_{H^1(B\cap B^\xi)}^2
\leq
C\bigl(\lambda_{1,\beta}(B^\xi)
-\lambda_{1,\beta}(B)\bigr),
\end{aligned}
\end{equation}
where $u_\xi$ and $u_B$ are the positive,
$L^2$-normalized first Robin eigenfunctions
of $B^\xi$ and $B$, respectively.
\end{theorem}

The estimate provides two forms of \(H^1\)-stability for the first Robin eigenfunctions: one after transporting \(u_\xi\) to the reference ball through \(\Psi_\xi\), and the other by comparing the eigenfunctions directly on \(B\cap B^\xi\).  In particular,
the latter estimate gives an improved exponent $1/2$ for the comparison on $B\cap B^\xi$.

\section{Notation and preliminary tools}
\label{sec:preliminaries}

In this section we introduce the variational framework used throughout the paper and fix the notation needed in the proof of the main result. We also collect the quantitative estimates that will be used later and study the radial first eigenfunction of the comparison ball. 
\subsection{The Robin eigenvalue on arbitrary open sets}
\label{sec:setting}

As anticipated in the introduction, our arguments require a notion of first Robin eigenvalue defined for arbitrary open sets of finite measure. We use the free-discontinuity relaxation introduced in
\cite{BucurGiacomini2010,BucurGiacomini2015}, following the notation of
\cite{AmbrosioFuscoPallara2000}.

A nonnegative measurable function $u$ belongs to
$SBV^{1/2}(\mathbb R^n)$ if

$$
 u\in L^2(\mathbb R^n),
 \qquad
 u^2\in SBV(\mathbb R^n).
$$

For $u\in SBV^{1/2}(\mathbb R^n)$, we set

$$
 J_u:=J_{u^2},
 \qquad
 u^\pm:=\sqrt{(u^2)^\pm},
 \qquad
 \nabla u:=\frac{\nabla(u^2)}{2u}\,\one_{\{u>0\}},
$$

and define the relaxed Robin energy by
\begin{equation}\label{eq:relaxed-energy}
\mathcal E_\beta(u)
:=
\int_{\mathbb R^n}|\nabla u|^2\,dx
+\beta\int_{J_u}
\bigl((u^+)^2+(u^-)^2\bigr)\,d\Hh^{n-1}.
\end{equation}

\begin{definition}\label{def:eigenvalue}
Let $\Omega\subset\mathbb R^n$ be open with $0<|\Omega|<\infty$. We define
\begin{equation}\label{eq:generalized-eigenvalue}
\lambda_{1,\beta}(\Omega)
:=
\min\left\{
\frac{\mathcal E_\beta(u)}
{\displaystyle\int_{\mathbb R^n}u^2\,dx}
\;:\;
\begin{array}{l}
u\in SBV^{1/2}(\mathbb R^n)\setminus\{0\},\\
u=0\quad \text{a.e. on }\mathbb R^n\setminus\Omega,\\
\Hh^{n-1}(J_u\setminus\partial\Omega)=0
\end{array}
\right\}.
\end{equation}
This problem admits a non-negative minimizer (see \cite{BucurGiacomini2010,BucurGiacomini2015}), called first Robin eigenfunction. 
\end{definition}

For bounded Lipschitz sets, this definition agrees with the classical Robin eigenvalue introduced in the previous section. Hence, we use the same notation
$\lambda_{1,\beta}(\Omega)$ in both settings.

In what follows, we fix a nonnegative, \(L^2\)-normalized first Robin eigenfunction \(u_\Omega\).

We point out that the admissible functions in \eqref{eq:generalized-eigenvalue} are defined on the whole space and vanish almost everywhere outside $\Omega$. With this convention,
\begin{equation}\label{eq:eigenfunction-energy}
\mathcal E_\beta(u_\Omega)=\lambda_{1,\beta}(\Omega).
\end{equation}

From now on, the volume is fixed and equal to $V>0$. Let $R>0$ be such that
$|B_R|=V$, and set

$$
 \lambda_B:=\lambda_{1,\beta}(B_R),
 \qquad
 \delta_\beta(\Omega)
 :=\lambda_{1,\beta}(\Omega)-\lambda_B.
$$

For every $\Omega$ with $|\Omega|=V$, we fix a ball $B=B_R(x_\Omega)$ achieving the Fraenkel asymmetry of $\Omega$ in the sense of \autoref{def:fraenkel-ball}.

\subsection{Quantitative Bossel--Daners  inequality with Fraenkel asymmetry}
\label{sec:inputs}

We now recall the quantitative Robin Faber--Krahn
inequality and some immediate consequences.

\begin{theorem}[{\cite[Theorem~1.1]{BucurFeroneNitschTrombetti2018}}]\label{thm:qfk}
For every $n\ge2$, $\beta>0$, and $V>0$, there exists
$c=c(n,\beta,V)>0$ such that, for every open set
$\Omega\subset\mathbb R^n$ with $|\Omega|=V$,
\begin{equation}\label{eq:qfk}
 \delta_\beta(\Omega)\ge c\,\mathcal A(\Omega)^2.
\end{equation}
\end{theorem}

 Thus if $B$ is a ball that achieves the Fraenkel asymmetry of $\Omega$ in the sense of \autoref{def:fraenkel-ball}, \autoref{thm:qfk} immediately gives
\begin{equation*}
 |\Omega\triangle B|
 \le C\,\delta_\beta(\Omega)^{1/2},
\end{equation*}
and, since $|\Omega|=|B|$,
\begin{equation}\label{eq:outer-measure}
 |\Omega\setminus B|
 =|B\setminus\Omega|
 \le C\,\delta_\beta(\Omega)^{1/2}.
\end{equation}

We now recall another result that will be useful in the following. 

\begin{lemma}\label{lem:linfty}
Let \(n\geq2\), \(\beta>0\), and \(V,\Lambda>0\). There exists a constant
\(C=C(n,\beta,V,\Lambda)>0\) such that, for every open set
\(\Omega\subset\mathbb R^n\) with
\[
    |\Omega|=V,
    \qquad
    \lambda_{1,\beta}(\Omega)\leq\Lambda,
\]
every nonnegative \(L^2\)-normalized first eigenfunction satisfies
\[
    \|u_\Omega\|_{L^\infty(\mathbb R^n)}
    \leq C.
\]
\end{lemma}
\begin{proof}
The \(L^\infty\) estimate follows from \cite[Lemma 2.3]{BucurFeroneNitschTrombetti2018}, whose proof is based on the truncation argument of \cite[Theorem 6.11]{BucurGiacomini2015}.We recall the main steps of the argument, keeping track of the dependence of the constants on the parameters.

Recall that
\[
\lambda_\Omega\leq\Lambda,
\]
and that
\[
\|u_\Omega\|_{L^2(\mathbb R^n)}=1,
\qquad
\mathcal E_\beta(u_\Omega)=\lambda_\Omega.
\]
For \(M>0\), as in \cite[(6.30)]{BucurGiacomini2015}, let
\[
u_\Omega\wedge M:=\min\{u_\Omega,M\},
\qquad
u_{\Omega,M}:=(u_\Omega-M)_+,
\qquad
\alpha(M):=|\{u_\Omega>M\}|.
\]
The normalization immediately gives
\begin{equation}\label{bound alpha}
    \alpha(M)\leq M^{-2}.
\end{equation}

Using $u_\Omega\wedge M$ as a test function in \eqref{eq:generalized-eigenvalue} and following \cite[(6.30)--(6.34)]{BucurGiacomini2015}, we have
\begin{equation}\label{enery u_M bound}
\mathcal E_\beta(u_{\Omega,M})
\leq
\lambda_\Omega\int_{\{u_\Omega>M\}}(u_\Omega^2-M^2)\,dx
\leq
\Lambda\int_{\{u_\Omega>M\}}(u_\Omega^2-M^2)\,dx.
\end{equation}

Following
\cite[(6.35)-(6.36)]{BucurGiacomini2015} and using the scaling relation \cite[(6.6)]{BucurGiacomini2015} and
\cite[Lemma 6.7(d)]{BucurGiacomini2015}, there exist \(m_0>0\) and
\(c_0>0\), depending only on \(n\) and \(\beta\), such that
\[
\mathcal E_\beta(u_{\Omega,M})
\geq
c_0\alpha(M)^{-1/n}
\int_{\mathbb R^n}u_{\Omega,M}^2\,dx
\]
whenever \(\alpha(M)\leq m_0\).

By \eqref{bound alpha}, the latter condition holds for all sufficiently large \(M\), with a threshold depending only on \(n\) and \(\beta\). Combining this estimate with \eqref{enery u_M bound}, and using
\[
u_\Omega^2-M^2=u_{\Omega,M}^2+2Mu_{\Omega,M}
\qquad\text{on }\{u_\Omega>M\},
\]
gives
\[
\left(c_0-\Lambda\alpha(M)^{1/n}\right)
\int_{\mathbb R^n}u_{\Omega,M}^2\,dx
\leq
2\Lambda M\alpha(M)^{1/n}
\int_{\mathbb R^n}u_{\Omega,M}\,dx.
\]
Using again \eqref{bound alpha}, we may choose
\[
M_0=M_0(n,\beta,\Lambda)
\]
so that, for every \(M\geq M_0\),
\[
c_0-\Lambda\alpha(M)^{1/n}\geq \frac{c_0}{2}.
\]
Consequently,
\begin{equation}\label{l2 l1 um}
    \int_{\mathbb R^n}u_{\Omega,M}^2\,dx
\leq
C M\alpha(M)^{1/n}
\int_{\mathbb R^n}u_{\Omega,M}\,dx,
\end{equation}

where \(C=C(n,\beta,\Lambda)\).

We now follow the argument leading to \cite[(6.37)]{BucurGiacomini2015}. Set
\[
f(M):=\int_{\mathbb R^n}(u_\Omega-M)_+\,dx.
\]
By Hölder's inequality and \eqref{l2 l1 um},
\[
f(M)^2
\leq
\alpha(M)\int_{\mathbb R^n}u_{\Omega,M}^2\,dx
\leq
C M\alpha(M)^{1+1/n}f(M).
\]
Hence
\[
f(M)\leq C M\alpha(M)^{1+1/n},
\]
or equivalently,
\begin{equation}\label{bound f}
\alpha(M)f(M)^{-n/(n+1)}
\geq
cM^{-n/(n+1)},
\end{equation}
with $c=c(n,\beta,\Lambda)$

Moreover, the layer-cake formula gives
\[
f(M)
=
\int_{\mathbb R^n}(u_\Omega-M)_+\,dx
=
\int_M^\infty |\{u_\Omega>t\}|\,dt
=
\int_M^\infty\alpha(t)\,dt.
\]
Thus \(f\) is absolutely continuous and
\[
f'(M)=-\alpha(M)
\]
for almost every \(M>0\). 

Therefore, for every
\[
M_0<T<\|u_\Omega\|_{L^\infty(\mathbb R^n)},
\]
from \eqref{bound f} we obtain
\begin{equation}\label{dis integral}
    c\int_{M_0}^{T}M^{-n/(n+1)}\,dM
\leq
-\int_{M_0}^{T}
f'(M)f(M)^{-n/(n+1)}\,dM.
\end{equation}
By explicit computations, 
\[
-\int_{M_0}^{T}
f'(M)f(M)^{-n/(n+1)}\,dM
=
(n+1)
\left(
f(M_0)^{1/(n+1)}
-
f(T)^{1/(n+1)}
\right)
\leq
(n+1)f(M_0)^{1/(n+1)}.
\]
On the other hand,
\[
f(M_0)
\leq
\int_{\mathbb R^n}u_\Omega\,dx
\leq
|\{u_\Omega>0\}|^{1/2}\|u_\Omega\|_2
\leq
V^{1/2}.
\]
Thus from \eqref{dis integral} we deduce
\[
T^{1/(n+1)}
\leq
M_0^{1/(n+1)}
+
C(n,\beta,\Lambda)V^{1/(2(n+1))}.
\]
Letting \(T\) approach \(\|u_\Omega\|_{L^\infty(\mathbb R^n)}\), we conclude that
\[
\|u_\Omega\|_{L^\infty(\mathbb R^n)}
\leq
C(n,\beta,V,\Lambda).
\]
\end{proof}

\begin{corollary}\label{cor:outside-mass}
Let $n\geq 2$, $\beta>0$, $V>0$, and $\delta_0>0$.
There exists a constant $C=C(n,\beta,V,\delta_0)>0$
such that the following holds for every open set
$\Omega\subset\mathbb R^n$ with $|\Omega|=V$ and
$\delta_\beta(\Omega)\leq\delta_0$.

Let $B$ be any ball achieving the Fraenkel asymmetry
of $\Omega$, in the sense of \autoref{def:fraenkel-ball}. Then, for every nonnegative, $L^2$-normalized
first Robin eigenfunction $u_\Omega$ on $\Omega$,
extended by zero outside $\Omega$, one has
\begin{equation}\label{eq:outside-mass}
\int_{B^c}u_\Omega^2\,dx
\leq C\,\delta_\beta(\Omega)^{1/2}.
\end{equation}
\end{corollary}
\begin{proof}
Using \autoref{lem:linfty}, with $\Lambda=\lambda_B+\delta_0$, we deduce that there exists a constant
depending only on $n$, $\beta$, $V$ and $\delta_0$ such that 
\[
 \int_{B^c}u_\Omega^2\,dx
 =\int_{\Omega\setminus B}u_\Omega^2\,dx
 \le \|u_\Omega\|_\infty^2|\Omega\setminus B|.
\]
The conclusion follows from \eqref{eq:outer-measure}.
\end{proof}

\subsection{Some properties of the first eigenfunction of a ball}
\label{sec:radial}

Since the eigenvalue deficit, the Fraenkel asymmetry, and the eigenfunction distance are invariant under simultaneous translations of $\Omega$ and $B$, we may assume throughout Sections~3--6 that the selected ball achieving the Fraenkel asymmetry is \(B=B_R(0)\).

Let $u_B$ be the nonnegative
first eigenfunction normalized in $L^2(B)$. By simplicity
and rotational invariance, there exists a smooth and decreasing function $\phi:[0,R]\to [0,+\infty)$ such that
\[
 u_B(x)=\phi(|x|)\one_{B_R}(x).
\]
Moreover, $\phi$ solves
\begin{equation}\label{eq:radial-ode}
 \begin{cases}
 -\phi''(r)-\dfrac{n-1}{r}\phi'(r)=\lambda_B\phi(r)
     & 0<r<R,\\[1.5mm]
 \phi'(0)=0,\\[1mm]
 \phi'(R)+\beta\phi(R)=0.
 \end{cases}
\end{equation}
It is well known that $\phi$ is strictly decreasing and $\phi'(0)=0$. Thus $\phi$ is  bounded both from above and from below by positive constants, i.e.
\begin{equation}\label{eq:phi-bounds}
0<m_B:=\phi(R)
\leq \phi(|x|)
\leq \phi(0)=:M_B<+\infty.
\end{equation}
We introduce
\begin{equation}\label{eq:q-def}
 q(r):=-\frac{\phi'(r)}{\phi(r)},
 \qquad 0\le r\le R.
\end{equation}
The function $q$ is the opposite of the derivative of $\log(\phi)$. This quantity arises naturally in the study of the Robin eigenvalue and has already appeared in several approaches to the problem. In particular, the quantity 
\[\dfrac{|\nabla u_B(x)|}{u_B(x)}=q(|x|),\qquad x\in B,\] 
plays a distinguished role in the $H$-function method underlying the Bossel--Daners inequality \cite{Bossel1986,Daners2006}.

We now recall some elementary properties of $q$ that will be used repeatedly. In particular we recall some properties proved in \cite[Section~4.2]{BW}, see also \cite[Proposition 3.5]{CV}.

\begin{lemma}\label{lem:q}
The function $q$ satisfies
\begin{equation}\label{eq:riccati}
 q'(r)
 =
 \lambda_B+q(r)^2-\frac{n-1}{r}q(r),
 \qquad r\in(0,R],
\end{equation}
and
\begin{equation}\label{eq:q-less-beta}
 q(0)=0,
 \qquad
 0<q(r)<\beta
 \quad\text{for }r\in(0,R),
 \qquad
 q(R)=\beta.
\end{equation}
Moreover, $q$ is strictly increasing on $[0,R]$ and
\begin{equation}\label{eq:q-boundary-derivative}
 q'(R)>0.
\end{equation}
Finally, there exists a constant $c_0=c_0(n,\beta,R)>0$ such that
\begin{equation}\label{eq:q-linear}
 \beta-q(r)\geq c_0(R-r)
 \qquad\text{for every }0\leq r\leq R.
\end{equation}
\end{lemma}

\begin{proof}
The properties \eqref{eq:riccati}--\eqref{eq:q-boundary-derivative} are
proven in \cite[Section~4.2]{BW}.

It remains only to prove \eqref{eq:q-linear}. Define
\[
 h(r):=\frac{\beta-q(r)}{R-r},
 \qquad 0\leq r<R.
\]
Since $q(R)=\beta$ and $q$ is $C^1$ up to $R$, the function $h$ extends
continuously to $r=R$, with
\[
 h(R)=q'(R)>0.
\]
Moreover, \eqref{eq:q-less-beta} gives $h(r)>0$ for every $r\in[0,R)$.
Therefore the continuous extension of $h$ is strictly positive on the
compact interval $[0,R]$, and hence
\[
 c_0:=\min_{[0,R]}h>0,
\]
which proves \eqref{eq:q-linear}.
\end{proof}
We now introduce the vector field that will be used in the following. Recall that, in $B$,
\[
    \nabla\log\phi(|x|)
    =
    -q(|x|)\frac{x}{|x|}.
\]
Since $q(R)=\beta$, we have
\[
    \nabla\log\phi(|x|)
    =
    -\beta\frac{x}{|x|}
    \qquad\text{on }\partial B.
\]
We choose an extension $\bar\phi$ of $\phi$ to $\mathbb R^n$ such that the latter equality holds with $\phi$ replaced by $\bar\phi$ on $B^c$. In other words, we set
\begin{equation}\label{eq:phibar}
\bar\phi(x)
:=
\begin{cases}
\phi(|x|), & |x|\leq R,\\[1mm]
\phi(R)e^{-\beta(|x|-R)}, & |x|>R.
\end{cases}
\end{equation}
By the Robin boundary condition $\phi'(R)=-\beta\phi(R)$, we have that
$\bar\phi\in C^1(\mathbb R^n)$.
We then set
\[
    b:=\nabla\log\bar\phi.
\]
Thus $b\in  W^{1,\infty}(\R^n)$ and
\[
    b(x)=
    \begin{cases}
    -q(|x|)\dfrac{x}{|x|}, & 0<|x|<R,\\[2mm]
    -\beta\dfrac{x}{|x|}, & |x|>R,
    \end{cases}
\]
with $b(0)=0$. In particular, \autoref{lem:q} gives
\begin{equation}\label{bound b}
       |b|\leq\beta
    \qquad\text{in }\mathbb R^n.
\end{equation}

Inside the ball, by direct computation and using \eqref{eq:riccati} (see  \cite[Proposition~2.3]{BucurDaners2010}) we have
\begin{equation}\label{eq:lambda as b}
      -\operatorname{div}b-|b|^2
    =
    \lambda_B
    \qquad\text{in }B,
\end{equation}

 On the other hand, a direct computation in $B^c$
gives
\begin{equation}\label{boumd b in Bc}
     -\operatorname{div}b-|b|^2
    =
    \frac{(n-1)\beta}{|x|}-\beta^2
    \geq-\beta^2.
\end{equation}

Finally, from \autoref{lem:q}, for every $x\in B$,
\begin{equation}\label{eq:b-distance}
    \beta-|b(x)|
    =
    \beta-q(|x|)
    \geq
    c_0\,\operatorname{dist}(x,\partial B).
\end{equation}

\section{Reformulation of the result}
\label{sec:reformulation}

We reformulate the eigenfunction stability problem in terms of the quotient
between $u_\Omega$ and the first eigenfunction of the comparison ball.

Let $\Omega\subset\mathbb R^n$ be an open set with $|\Omega|=V$,
and let $B$ be a ball achieving the Fraenkel asymmetry of
$\Omega$ in the sense of \autoref{def:fraenkel-ball}.
Let $\bar\phi$ be the positive function introduced in
\autoref{sec:radial}. For every function $u$ admissible in
\eqref{eq:generalized-eigenvalue} with
$\mathcal E_\beta(u)<+\infty$, we set
\begin{equation}\label{eq:f-def}
    f:=\frac{u}{\bar\phi}.
\end{equation}
In particular, for a nonnegative, $L^2$-normalized first
Robin eigenfunction $u_\Omega$, we write
\[
    f_\Omega:=\frac{u_\Omega}{\bar\phi}.
\]

If $\Omega=B$, then $u_\Omega$ and $\bar\phi$ are proportional and $f_\Omega$ is
constant on $B$. Therefore, in order to compare $u_\Omega$ with the
eigenfunction of the ball, it is natural to estimate how far $f_\Omega$ is from
being constant on $B$.

We notice that, since $\frac{1}{\bar\phi}$ is strictly positive and $C^1$, we have

$$
    f\in SBV^{1/2}_{loc}(\mathbb R^n).
$$

Moreover $J_f=J_u$ and
\begin{equation}\label{eq:f-traces}
f^\pm=\frac{u^\pm}{\bar\phi}.
\end{equation}

We denote by
\begin{equation}\label{eq:f-mean}
\langle f\rangle_B:=\dfrac{1}{|B|}\int_B f\,dx
\end{equation}
the average of $f$ over $B$. We have the following result. 

\begin{proposition}\label{prop:reduction}
Let $\delta_0>0$. Suppose that there exists a constant \(C_0=C_0(n,\beta,V,\delta_0)>0\) such that, for every open set \(\Omega\subset\mathbb R^n\) with $|\Omega|=V$ and
$\delta_\beta(\Omega)\le\delta_0$, every ball $B$ achieving the Fraenkel asymmetry of $\Omega$ in the sense of \autoref{def:fraenkel-ball}, and every nonnegative, $L^2$-normalized first Robin eigenfunction $u_\Omega$, the quotient $f_\Omega$ defined above satisfies
\begin{equation}\label{eq:target-oscillation}
        \int_B |f_\Omega-\langle f_\Omega\rangle_B|^2\,dx
    \le C_0\delta_\beta(\Omega)^{1/2},
\end{equation}

Then there exists a constant $C=C(n,\beta,V,\delta_0)$ such that 
\[
    \|u_\Omega-u_B\|_{L^2(\mathbb R^n)}
    \le C\delta_\beta(\Omega)^{1/4}.
\]
\end{proposition}

\begin{proof}
Let $m=\langle f_\Omega\rangle_B$.  Since $f_\Omega\ge0$, one has $m\ge0$.  On $B$,
$u_\Omega=\phi f_\Omega$, and therefore \eqref{eq:phi-bounds} and
\eqref{eq:target-oscillation} give
\begin{equation}\label{eq:reduction-inside}
 \int_B|u_\Omega-m\phi|^2\,dx
 \le M_B^2\int_B|f_\Omega-m|^2\,dx
 \le C\delta_\beta(\Omega)^{1/2}.
\end{equation}
On $B^c$, the zero extension of $u_B$ vanishes.  \autoref{cor:outside-mass} consequently gives
\begin{equation}\label{eq:reduction-multiple}
 \|u_\Omega-mu_B\|_{L^2(\mathbb R^n)}^2
 \le C\delta_\beta(\Omega)^{1/2}.
\end{equation}
  Since $m\ge0$ and both $u_\Omega$ and $u_B$ are normalized in $L^2(\R^n)$, the reverse
triangle inequality gives
\[
 |1-m|
 =\bigl|\|u_\Omega\|_2-\|mu_B\|_2\bigr|
 \le\|u_\Omega-mu_B\|_2
\leq C\delta_\beta(\Omega)^{\frac{1}{4}}.\]
It follows that
\[
 \|u_\Omega-u_B\|_2
 \le\|u_\Omega-mu_B\|_2+|m-1|
 \le C\delta_\beta^{1/4}(\Omega),
\]
which proves the claim.
\end{proof}

Thus, the proof of \autoref{thm:main} is reduced to showing that
$u_\Omega/\bar\phi$ is quantitatively close to a constant on $B$.

\section{Rewriting the energy}
\label{sec:groundstate}

We now rewrite the Robin energy in \eqref{eq:relaxed-energy} in terms of the quotient
\[
    f=\frac{u}{\bar\phi}
\]
introduced in \eqref{eq:f-def}. 

\begin{proposition}\label{prop:groundstate}
Let $u\in SBV^{1/2}(\mathbb R^n)$ satisfy
$\mathcal E_\beta(u)<+\infty$. Let $\bar\phi$ and $b$ be the functions introduced in \autoref{sec:radial}, and set
\[
    f:=\frac{u}{\bar\phi}.
\]
Then
\begin{align}
 \mathcal E_\beta(u)
 ={}&
 \int_{\mathbb R^n}\bar\phi^2|\nabla f|^2\,dx
 +\int_{\mathbb R^n}
 \bigl(-\operatorname{div}b-|b|^2\bigr)u^2\,dx
 \label{eq:groundstate}\\
 &+
 \int_{J_u}
 \Bigl[
 (\beta-b\cdot\nu_u)(u^+)^2
 +(\beta+b\cdot\nu_u)(u^-)^2
 \Bigr]\,d\Hh^{n-1}.
 \notag
\end{align}
\end{proposition}

\begin{proof}
On the approximate differentiability set of $u$, using the notation introduced in \autoref{sec:radial},
\[
    \nabla u
    =\bar\phi\nabla f+f\nabla\bar\phi
    =\bar\phi\nabla f+ub.
\]
Using
\[
    \nabla(u^2)=2u\nabla u,
\]
we obtain
\begin{equation}\label{eq:bulk-algebra}
    |\nabla u|^2
    =
    \bar\phi^2|\nabla f|^2
    +b\cdot\nabla(u^2)
    -|b|^2u^2
\end{equation}
almost everywhere.

Since $u^2\in SBV(\mathbb R^n)$,
\begin{equation}\label{eq:Dusq}
 D(u^2)
 =
 \nabla(u^2)\,\mathcal L^n
 +
 \bigl((u^+)^2-(u^-)^2\bigr)\nu_u\,
 \Hh^{n-1}\llcorner J_u.
\end{equation}
Using the integration-by-parts formula for $BV$ functions, we obtain
\begin{align}
 \int_{\mathbb R^n}b\cdot\nabla(u^2)\,dx
 ={}&
 -\int_{\mathbb R^n}(\operatorname{div}b)u^2\,dx
 \label{eq:ibp-u2}\\
 &-
 \int_{J_u}
 b\cdot\nu_u
 \bigl((u^+)^2-(u^-)^2\bigr)\,d\Hh^{n-1}.
 \notag
\end{align}
To be precise, since $b$ is not compactly supported, one first applies the formula to $\eta_k b$, where $\eta_k$ is a standard cutoff between $B_k$ and $B_{2k}$, and then lets $k\to\infty$. 

Substituting \eqref{eq:ibp-u2} into the integral of
\eqref{eq:bulk-algebra} and adding the Robin jump energy gives
\eqref{eq:groundstate}.
\end{proof}

We now apply \autoref{prop:groundstate} to the first eigenfunction $u_\Omega$.

\begin{proposition}\label{prop:coercivity-f} Let \(n\ge2\), \(\beta>0\), \(V>0\), and \(\delta_0>0\). Let \(\Omega\subset\mathbb R^n\) be open, with \(|\Omega|=V\) and \(\delta_\beta(\Omega)\le\delta_0\). Let $B$ be any ball achieving the Fraenkel asymmetry of $\Omega$ in the sense of \autoref{def:fraenkel-ball}, and let $u_\Omega$ be any nonnegative, $L^2$-normalized first Robin eigenfunction of $\Omega$. Let
\[
    f_\Omega=\frac{u_\Omega}{\bar\phi}.
\]
Then there exists a constant $C=C(n,\beta,V,\delta_0)>0$ such that
\begin{equation}\label{eq:coercivity-f}
 \int_B|\nabla f_\Omega|^2\,dx
 +
 \int_{J_{f_\Omega}\cap B}
 \operatorname{dist}(x,\partial B)
 \bigl((f_\Omega^+)^2+(f_\Omega^-)^2\bigr)
 \,d\Hh^{n-1}
 \leq
 C\,\delta_\beta(\Omega)^{1/2}.
\end{equation}
\end{proposition}

\begin{proof}
Since $u_\Omega$ is $L^2$-normalized,
\[
    \mathcal E_\beta(u_\Omega)
    =
    \lambda_{1,\beta}(\Omega)
    =
    \lambda_B+\delta_\beta(\Omega).
\]
Applying \autoref{prop:groundstate} to $u_\Omega$, we obtain
\begin{equation}\label{eqriscrittura lambda +delta}
    \begin{split}
 \lambda_B+\delta_\beta(\Omega)
 ={}&
 \int_{\mathbb R^n}\bar\phi^2|\nabla f_\Omega|^2\,dx
 +
 \int_{\mathbb R^n}
 \bigl(-\operatorname{div}b-|b|^2\bigr)u_\Omega^2\,dx\\
 &+
 \int_{J_{u_\Omega}}
 \Bigl[
 (\beta-b\cdot\nu_{u_\Omega})(u_\Omega^+)^2
 +
 (\beta+b\cdot\nu_{u_\Omega})(u_\Omega^-)^2
 \Bigr]\,d\Hh^{n-1}.
\end{split}
\end{equation}

We estimate separately the last two terms. 

\medskip

Firstly, by \eqref{eq:lambda as b} and \eqref{boumd b in Bc}, we have that 
\begin{equation}\label{inermediate 1 per cos}
\int_{\mathbb R^n}
 \bigl(-\operatorname{div}b-|b|^2\bigr)u_\Omega^2\,dx\ge\lambda_B\int_Bu_\Omega^2\,dx
 -\beta^2\int_{B^c}u_\Omega^2\,dx=\lambda_B-(\lambda_B+\beta^2)\int_{B^c} u_\Omega^2\,dx,
\end{equation}
where we have used that 
\[
    \int_Bu_\Omega^2\,dx
    =
    1-\int_{B^c}u_\Omega^2\,dx.
\]
\medskip

We are now ready to bound the jump term in \eqref{eqriscrittura lambda +delta}. By \eqref{bound b},
\[
    \beta\mp b\cdot\nu_{u_\Omega}
    \geq \beta-|b|
    \geq0,
\]
so the jump contribution is nonnegative on the whole space. Moreover, on $B$, \eqref{eq:b-distance} gives
\[
\begin{aligned}
 &(\beta-b\cdot\nu_{u_\Omega})(u_\Omega^+)^2
 +(\beta+b\cdot\nu_{u_\Omega})(u_\Omega^-)^2\geq
 c_0\,\operatorname{dist}(x,\partial B)
 \bigl((u_\Omega^+)^2+(u_\Omega^-)^2\bigr).
\end{aligned}
\]
Thus we have that 
\[
\begin{aligned}
 &\int_{J_{u_\Omega}}
 \Bigl[
 (\beta-b\cdot\nu_{u_\Omega})(u_\Omega^+)^2
 +
 (\beta+b\cdot\nu_{u_\Omega})(u_\Omega^-)^2
 \Bigr]\,d\Hh^{n-1}\\
 &\qquad\geq c_0 \int_{J_{u_\Omega}\cap B}\operatorname{dist}(x,\partial B)
 \bigl((u_\Omega^+)^2+(u_\Omega^-)^2\bigr)\,d\Hh^{n-1}.
\end{aligned}
\]

Therefore, using the latter and \eqref{inermediate 1 per cos} in \eqref{eqriscrittura lambda +delta}, we have that 
\begin{align}
 \delta_\beta(\Omega)
 \geq{}&
 \int_{\mathbb R^n}\bar\phi^2|\nabla f_\Omega|^2\,dx
 -(\lambda_B+\beta^2)
 \int_{B^c}u_\Omega^2\,dx
 \label{eq:coercivity-u}\\
 &+
 c_0
 \int_{J_{u_\Omega}\cap B}
 \operatorname{dist}(x,\partial B)
 \bigl((u_\Omega^+)^2+(u_\Omega^-)^2\bigr)
 \,d\Hh^{n-1}.
 \notag
\end{align}

By \autoref{cor:outside-mass},
\[
    \int_{B^c}u_\Omega^2\,dx
    \leq C\,\delta_\beta(\Omega)^{1/2}.
\]
Since $\delta_\beta(\Om)\leq\delta_0$, \eqref{eq:coercivity-u} therefore implies
\begin{align*}
 \int_B\bar\phi^2|\nabla f_\Omega|^2\,dx
 &+
 \int_{J_{u_\Omega}\cap B}
 \operatorname{dist}(x,\partial B)
 \bigl((u_\Omega^+)^2+(u_\Omega^-)^2\bigr)
 \,d\Hh^{n-1}\leq C\,\delta_\beta(\Omega)^{1/2}.
\end{align*}
Finally, inside $B$ we have $\bar\phi=\phi$ and
\[
    u_\Omega^\pm=\phi f_\Omega^\pm
    \qquad\text{on }J_{f_\Omega}=J_{u_\Omega}.
\]
Using the two-sided bounds \eqref{eq:phi-bounds} for $\phi$, we obtain
\eqref{eq:coercivity-f}.
\end{proof}

Combining \autoref{prop:coercivity-f} with \autoref{prop:reduction}, the proof of \autoref{thm:main} is reduced to a Poincar\'e inequality controlling
\[
    \int_B |f_\Omega-\langle f_\Omega\rangle_B|^2\,dx
\]
by the two terms appearing on the left-hand side of
\eqref{eq:coercivity-f}.

\section{The Poincar\'e estimate}
\label{sec:weighted-poincare}

The aim of this section is to prove \eqref{eq:target-oscillation}. More precisely, we prove the following result. 
\begin{theorem}\label{thm:poincare-estimate}
Let $\beta>0$, $n\ge2$, and $V>0$. There exist
$\delta_0=\delta_0(n,\beta,V)>0$ and
$C=C(n,\beta,V)>0$ such that, for every open set
$\Omega\subset\mathbb R^n$ satisfying
\[
    |\Omega|=V,
    \qquad
    \delta_\beta(\Omega)\le\delta_0,
\]
for every ball $B$ achieving the Fraenkel asymmetry of $\Omega$ in the sense of \autoref{def:fraenkel-ball}, and every nonnegative, $L^2$-normalized first Robin eigenfunction $u_\Omega$, the quotient $f_\Omega=u_\Omega/\bar\phi$ satisfies
\begin{equation}\label{eq:poincare-estimate}
    \int_B
    |f_\Omega-\langle f_\Omega\rangle_B|^2\,dx
    \le
    C\,\delta_\beta(\Omega)^{1/2}.
\end{equation}
\end{theorem}

In order to prove this result we need the following lemma.

\begin{lemma}\label{lem:weighted-BV-poincare}
Let $B=B_R\subset\R^n$. There exists $C=C(n)>0$ such that, for every
$w\in BV(B)$,
\begin{equation}\label{eq:weighted-BV-poincare}
    \inf_{c\in\R}\int_B|w-c|\,dx
    \leq
    C
    \int_B
    \operatorname{dist}(x,\partial B)\,d|Dw|(x).
\end{equation}
\end{lemma}

\begin{proof}
By
\cite[Theorem~3.2]{DRELICHMAN2008286}, for every locally Lipschitz $w\in L^1(B)$
\begin{equation} \label{poinc weighted in L1}  \inf_{c\in\R}\int_B|w-c|\,dx
    \leq
    C
    \int_B
    \operatorname{dist}(x,\partial B)|\nabla w|\,dx.
\end{equation}
Let now $w\in BV(B)$. By the  approximation theorem
\cite[Theorem~3.9]{AmbrosioFuscoPallara2000}, there exists
$w_j\in C^\infty(B)\cap W^{1,1}(B)$ such that
\[
    w_j\to w
    \quad\text{in }L^1(B),
    \qquad
    \int_B|\nabla w_j|\,dx\longrightarrow |Dw|(B).
\]
 By Reshetnyak's continuity
theorem \cite[Theorem~2.39]{AmbrosioFuscoPallara2000},
\[
    \int_B
    \operatorname{dist}(x,\partial B)|\nabla w_j|\,dx
    \longrightarrow
    \int_B
    \operatorname{dist}(x,\partial B)\,d|Dw|.
\]

Moreover, 
\[
    \left|
    \inf_{c\in\R}\|w_j-c\|_{L^1(B)}
    -
    \inf_{c\in\R}\|w-c\|_{L^1(B)}
    \right|
    \leq
    \|w_j-w\|_{L^1(B)}
    \longrightarrow0.
\]
Hence, passing to the limit in \eqref{poinc weighted in L1} proves \eqref{eq:weighted-BV-poincare}.
\end{proof}

\begin{remark}\label{rem:minimizing-constants}
We recall the following elementary properties of \(L^1\) and \(L^2\) minimization over constants.

First, if $g\in L^2(B)$, then for every $c\in\R$,
\begin{equation}\label{eq:L2-best-constant}
    \int_B|g-c|^2\,dx
    =
    \int_B|g-\langle g\rangle_B|^2\,dx
    +
    |B|\,|c-\langle g\rangle_B|^2.
\end{equation}
In particular, $\langle g\rangle_B$ minimizes the $L^2$ distance of
$g$ from the constants.

Second, if $w\in L^1(B)$ satisfies
\begin{equation}\label{eq:zero-median-condition}
    |\{w>0\}|\leq\frac{|B|}{2},
    \qquad
    |\{w<0\}|\leq\frac{|B|}{2},
\end{equation}
then
\begin{equation}\label{eq:L1-zero-minimizer}
    \int_B|w|\,dx
    =
    \inf_{c\in\R}\int_B|w-c|\,dx.
\end{equation}
Indeed, if $c>0$, then
\[
\begin{aligned}
\int_B\bigl(|w-c|-|w|\bigr)\,dx
&\geq
c|\{w\leq0\}|-c|\{w>0\}|
\\
&=
c\bigl(|B|-2|\{w>0\}|\bigr)
\geq0.
\end{aligned}
\]
The case $c<0$ follows in the same way, using the second inequality in
\eqref{eq:zero-median-condition}. Since equality is attained for $c=0$,
\eqref{eq:L1-zero-minimizer} follows.
\end{remark}

\begin{proof}[Proof of \autoref{thm:poincare-estimate}]
By \autoref{lem:linfty} and \eqref{eq:phi-bounds}, after possibly reducing
$\delta_0$, we have
\begin{equation}\label{eq:f-uniform-bound}
    0\leq f_\Omega\leq C
    \qquad\text{a.e. in }B,
\end{equation}
where, here and in the following, $C$ denotes a positive constant
depending only on $n,\beta, V$ and $\delta_0$.

We first observe that by \eqref{eq:outside-mass} and the normalization of
$u_\Omega$,
\[
    \int_B u_\Omega^2\,dx
    =
    1-\int_{B^c}u_\Omega^2\,dx
    \geq
    1-C\delta_\beta(\Omega)^{1/2}.
\]

Using $u_\Omega=\phi f_\Omega$ in $B$, together with
\eqref{eq:phi-bounds} and \eqref{eq:f-uniform-bound}, we obtain that for $\delta_0$ sufficiently small there exists a constant $c=c(n,\beta,V,\delta_0)$ such that 
\[
    c
    \leq
    \int_B f_\Omega^2\,dx
    \leq
    C\int_B f_\Omega\,dx.
\]
Therefore
\begin{equation}\label{eq:f-average-lower}
    \langle f_\Omega\rangle_B\geq c>0.
\end{equation}
Fix $\theta>0$, depending only on $n,\beta$ and $V$, sufficiently small
so that
\[
    \langle f_\Omega\rangle_B\geq4\theta,
\]
and set
\[
    h:=\max\{f_\Omega,\theta\}.
\]
Since $f_\Omega\in SBV^{1/2}(B)$, one has $f_\Omega^2\in SBV(B)$, hence the $BV$ chain rule
\cite[Theorem~3.99]{AmbrosioFuscoPallara2000} (see  \cite[Lemma 3.2]{BucurGiacomini2010} for more details) gives
$h\in SBV(B)$.

On $\{f_\Omega<\theta\}$,
\eqref{eq:f-average-lower} gives
\[
    |f_\Omega-\langle f_\Omega\rangle_B|
    \geq
    \langle f_\Omega\rangle_B-f_\Omega
    \geq3\theta.
\]
Thus, squaring and integrating, we have
\[
    |\{f_\Omega<\theta\}|
    \leq\frac{1}{9\theta^2}
    \int_{\{f_\Omega<\theta\}}|f_\Omega-\langle f_\Omega\rangle_B|^2\,dx\leq
    \frac{1}{9\theta^2}
    \int_B|f_\Omega-\langle f_\Omega\rangle_B|^2\,dx.
\]
Since $0\leq h-f_\Omega\leq\theta$ on $\{f_\Omega<\theta\}$ and $h=f_\Omega$ elsewhere,
\begin{equation}\label{eq:truncation-error}
    \|h-f_\Omega\|_{L^2(B)}
    \leq \theta|\{f_\Omega<\theta\}|^{\frac{1}{2}}\leq
    \frac13
    \|f_\Omega-\langle f_\Omega\rangle_B\|_{L^2(B)}.
\end{equation}

We now choose $a\in\R$ such that
\begin{equation}\label{eq:half-level}
    |\{h>a\}|\leq\frac{|B|}{2},
    \qquad
    |\{h<a\}|\leq\frac{|B|}{2}.
\end{equation}
Since $\theta\leq h\leq C$,
\[
    \theta\leq a\leq C.
\]
By \eqref{eq:L2-best-constant},
\[
    \|f_\Omega-\langle f_\Omega\rangle_B\|_{L^2(B)}
    \leq
    \|f_\Omega-a\|_{L^2(B)}.
\]
Using \eqref{eq:truncation-error}, we obtain
\[
\begin{aligned}
    \|f_\Omega-\langle f_\Omega\rangle_B\|_{L^2(B)}
    &\leq
    \|f_\Omega-h\|_{L^2(B)}
    +
    \|h-a\|_{L^2(B)}
\\
    &\leq
    \frac13
    \|f_\Omega-\langle f_\Omega\rangle_B\|_{L^2(B)}
    +
    \|h-a\|_{L^2(B)}.
\end{aligned}
\]
Therefore
\begin{equation}\label{eq:f-h-comparison}
    \int_B|f_\Omega-\langle f_\Omega\rangle_B|^2\,dx
    \leq
    C\int_B|h-a|^2\,dx.
\end{equation}

Hence, it remains to control the right-hand side. Consider
\[
    w:=(h-a)|h-a|.
\]
Since $h$ is bounded and belongs to $SBV(B)$, the $BV$ chain rule gives
$w\in SBV(B)$. Moreover,
\[
    |w|=|h-a|^2,
\]
and, by \eqref{eq:half-level},
\[
    |\{w>0\}|\leq\frac{|B|}{2},
    \qquad
    |\{w<0\}|\leq\frac{|B|}{2}.
\]
Thus \eqref{eq:L1-zero-minimizer} and
\autoref{lem:weighted-BV-poincare} give
\begin{equation}\label{eq:w-poincare}
    \int_B|h-a|^2\,dx
    \leq
    C\int_B
    \operatorname{dist}(x,\partial B)\,d|Dw|.
\end{equation}

For the absolutely continuous part of $Dw$,
\[
    |\nabla w|
    =
    2|h-a||\nabla h|
    \qquad\text{a.e. in }B.
\]
Therefore,
\[
\begin{aligned}
\int_B
\operatorname{dist}(x,\partial B)|\nabla w|\,dx
&\leq
2R\int_B|h-a||\nabla h|\,dx
\\
&\leq
2R
\left(\int_B|h-a|^2\,dx\right)^{1/2}
\left(\int_B|\nabla h|^2\,dx\right)^{1/2}.
\end{aligned}
\]

On the jump set,
\[
    w^\pm=(h^\pm-a)|h^\pm-a|.
\]
Since $\theta\leq h^\pm$ and $\theta\leq a\leq C$, we have
\[
    |w^\pm|
    =
    |h^\pm-a|^2
    \leq
    C(h^\pm)^2.
\]
Hence
\[
    |w^+-w^-|
    \leq
    C\bigl((h^+)^2+(h^-)^2\bigr).
\]
Using these two estimates in \eqref{eq:w-poincare} gives
\[
\begin{aligned}
    \int_B|h-a|^2\,dx
    \leq{}&
    C
    \left(\int_B|h-a|^2\,dx\right)^{1/2}
    \left(\int_B|\nabla h|^2\,dx\right)^{1/2}
\\
    &+
    C
    \int_{J_h\cap B}
    \operatorname{dist}(x,\partial B)
    \bigl((h^+)^2+(h^-)^2\bigr)\,d\Hh^{n-1}.
\end{aligned}
\]
Using Young's inequality in the first term of the righthand side,
\begin{equation}\label{eq:h-oscillation}
\begin{aligned}
    \int_B|h-a|^2\,dx
    \leq C\Bigg[
    &\int_B|\nabla h|^2\,dx
\\
    &+
    \int_{J_h\cap B}
    \operatorname{dist}(x,\partial B)
    \bigl((h^+)^2+(h^-)^2\bigr)\,d\Hh^{n-1}
    \Bigg],
\end{aligned}
\end{equation}
where the radius $R$ has been absorbed into the constant.

Finally,
\[
    |\nabla h|\leq|\nabla f_\Omega|
    \qquad\text{a.e. in }B,
\]
and $J_h\subset J_{f_\Omega}$. Moreover, At every point of \(J_h\), one of the following three cases occurs:
\begin{itemize}
    \item $h^+=f_\Omega^+$ and $h^-=f_\Omega^-$;
    \item $h^+=f_\Omega^+>\theta$ and $h^-=\theta$;
    \item $h^+=\theta$ and $h^-=f_\Omega^->\theta$.
\end{itemize}
In any case,

\[
    (h^+)^2+(h^-)^2
    \leq
    2\bigl((f_\Omega^+)^2+(f_\Omega^-)^2\bigr)
    \qquad\text{on }J_h.
\]
Combining \eqref{eq:f-h-comparison}, \eqref{eq:h-oscillation} and
\eqref{eq:coercivity-f}, we conclude that
\[
    \int_B|f_\Omega-\langle f_\Omega\rangle_B|^2\,dx
    \leq
    C\,\delta_\beta(\Omega)^{1/2}.
\]
This proves \eqref{eq:target-oscillation}.
\end{proof}

\section{Proof of the main result}\label{sec:proof-main}
We are now ready to prove \autoref{thm:main}.

\begin{proof}[Proof of \autoref{thm:main}]

Since 
\[\|u_\Omega-u_B\|_{L^2(\mathbb R^n)}
    \leq2,\]
it is sufficient to prove the thesis when $\delta_\beta(\Omega)$ is sufficiently small. 
Thus we assume that \(\delta_\beta(\Omega)\leq\delta_0\), where
\(\delta_0=\delta_0(n,\beta,V)>0\) is the constant given by
\autoref{thm:poincare-estimate}. By that result, the quotient
\[
    f_\Omega:=\frac{u_\Omega}{\bar\phi}
\]
satisfies
\[
    \int_B
    \bigl|f_\Omega-\langle f_\Omega\rangle_B\bigr|^2\,dx
    \leq
    C\,\delta_\beta(\Omega)^{1/2}.
\]
Therefore \eqref{eq:target-oscillation} holds, and
\autoref{prop:reduction} yields
\[
    \|u_\Omega-u_B\|_{L^2(\mathbb R^n)}
    \leq
    C\,\delta_\beta(\Omega)^{1/4}.
\]
We are only left to prove the sharpness of the exponent $1/4$, which can be proven by considering the family of volume-preserving ellipsoidal perturbations of the ball considered in \cite[Section 4]{BucurFeroneNitschTrombetti2018}. We provide details in \autoref{subsection sharpness}

\end{proof}

\subsection{Checking the sharpness}\label{subsection sharpness}

The exponent $1/4$ in \eqref{eq:main-estimate} is optimal,
even among smooth connected domains. This follows from
the ellipsoidal perturbations used in
\cite[Section~4]{BucurFeroneNitschTrombetti2018}
to establish the sharpness of the quantitative Robin
Faber--Krahn inequality. We provide the details for completeness.

Let $B_R$ be the ball centered at the origin with volume $V$,
and consider the family of ellipsoids

\[
E_t=T_tB_R,
\qquad
T_t=\operatorname{diag}(e^t,e^{-t},1,\ldots,1),
\]
for $|t|$ sufficiently small. Since $\det T_t=1$, we have $|E_t|=V$.

We first show that there exists a constant $c$ such that 
\begin{equation}\label{eq:ellipsoid-asymmetry}
\mathcal A(E_t)\geq c|t|.
\end{equation}
Since both $E_t$ is convex and centrally symmetric, the ball $B_R(0)$ achieves the Fraenkel asymmetry of $E_t$ in the sense of \autoref{def:fraenkel-ball}, and $\mathcal A(E_t)=|E_t\triangle B_R|/V$.

Writing $x=r\theta$, where $r=|x|$ and $\theta=x/|x|\in\mathbb S^{n-1}$ for $x\neq0$, we can express the ellipsoid as
\[
E_t=\left\{
r\theta:\theta\in\mathbb S^{n-1},
\ 0\leq r<r_t(\theta)
\right\},
\]
where
\[
r_t(\theta)
=\frac{R}{\sqrt{e^{-2t}\theta_1^2+
e^{2t}\theta_2^2+\sum_{i=3}^n\theta_i^2}}.
\]

Using the expansions of $e^{\pm2t}$, we obtain
\[
r_t(\theta)
=R+Rt(\theta_1^2-\theta_2^2)+O(t^2),
\]
uniformly in $\theta$. Integration in polar coordinates then gives
\[
\begin{aligned}
|E_t\triangle B_R|
&=\frac1n\int_{\mathbb S^{n-1}}
|r_t(\theta)^n-R^n|\,d\theta=R^n|t|\int_{\mathbb S^{n-1}}
|\theta_1^2-\theta_2^2|\,d\theta+O(t^2).
\end{aligned}
\]
Since the last integral is strictly positive, \eqref{eq:ellipsoid-asymmetry} follows.

We now prove that
\begin{equation}\label{eq:ellipsoid-deficit}
\delta_\beta(E_t)\leq Ct^2.
\end{equation}
The argument is analogous to the ellipsoidal construction
in \cite[Section~2.6]{BrascoDePhilippis2017}, with the
additional Robin boundary term.

Let $u_{B_R}$ be the positive,
$L^2$-normalized first Robin eigenfunction of $B_R$,
and use $v_t=u_{B_R}\circ T_t^{-1}$ as a test function
in \eqref{eq:intro-robin-quotient}. Since $\det T_t=1$,

\begin{equation}\label{eq:ellipsoid-rayleigh}
\begin{aligned}
\lambda_{1,\beta}(E_t)
&\leq
\int_{B_R}|T_t^{-T}\nabla u_{B_R}|^2\,dx+\beta\int_{\partial B_R}u_{B_R}^2
|T_t^{-T}\nu|\,d\mathcal H^{n-1}.
\end{aligned}
\end{equation}

We point out that the right-hand side is regular and even in $t$, since replacing
$t$ by $-t$ amounts to interchanging $\theta_1$
and $\theta_2$, which leaves the integrals
unchanged.

As a consequence, the first derivative of the highhandedness at $t=0$ 
vanish. Thus using Taylor's formula and that $E_0=B_R$,
we conclude that
\[
\lambda_{1,\beta}(E_t)
\leq\lambda_{1,\beta}(B_R)+Ct^2.
\]
Hence,

\[
0\leq\delta_\beta(E_t)\leq Ct^2,
\]
proving \eqref{eq:ellipsoid-deficit}.

Finally, applying \eqref{eq:intro-asymmetry-control}
to $E_t$ and using \eqref{eq:ellipsoid-asymmetry},
we obtain
\[
\|u_{E_t}-u_{B_R}\|_{L^2(\mathbb R^n)}^2
\geq c\,\mathcal A(E_t)
\geq c|t|.
\]

If \eqref{eq:main-estimate} held with an exponent
$\alpha>1/4$, then \eqref{eq:ellipsoid-deficit}
would imply
\[
c|t|^{1/2}
\leq C\delta_\beta(E_t)^\alpha
\leq C|t|^{2\alpha},
\]
which is impossible as $t\to0$, since
$2\alpha>1/2$. This proves the optimality
of the exponent $1/4$.

\section{Nearly spherical sets}
\label{sec:nearly-spherical}

We conclude by considering smooth perturbations of the ball. In this regime, we obtain a stronger stability estimate than in the general case.

Throughout this section, $B=B_R(0)$ is the ball of volume $V$,
\[
    \lambda_B:=\lambda_{1,\beta}(B),
\]
and $u_B$ denotes its positive first Robin eigenfunction, normalized in
$L^2(B)$.

In the following we will use the notations introduced in \autoref{def:nearly-spherical} for nearly spherical sets.

Moreover, we set
\[
\lambda_\xi:=\lambda_{1,\beta}(B^\xi)
\]
and denote by $u_\xi$ the positive first Robin eigenfunction of $B^\xi$,
normalized by
\[
\int_{B^\xi}u_\xi^2\,dx=1.
\]

The following local stability estimate holds.
\begin{lemma}\label{lem:local-coercivity}
Let $n\geq 2$, $\beta,V>0$, and
$\gamma\in(0,1)$.
There exist constants
$\eta_0=\eta_0(n,\beta,V,\gamma)>0$
and $c=c(n,\beta,V,\gamma)>0$
such that, for every
$B^\xi\in\mathcal S_{\eta_0,\gamma}(B)$,
\begin{equation}\label{eq:local-coercivity}
\lambda_{1,\beta}(B^\xi)
-\lambda_{1,\beta}(B)
\geq c\|\xi\|_{H^1(\partial B)}^2.
\end{equation}
\end{lemma}

The estimate follows from \cite[Theorem 1.3]{DL}, using the second variation of the first Robin eigenvalue computed in \cite[Section 7.2.3]{BW} and the improved continuity result of \cite[Proposition 4.4]{CV}.

We now recall the shape calculus tools used below; see \cite[Chapter 5]{HP18shape} for a general reference.

Let $\xi\in C^{2,\gamma}(\partial B)$ and let $h_\xi$ be the harmonic extension of $\xi/R$ to $B$, namely
\begin{equation}
\label{eq:h-xi}
\begin{cases}
    \Delta h_\xi=0 & \text{in }B,\\
    h_\xi=\xi/R & \text{on }\partial B.
\end{cases}
\end{equation}
We define
\begin{equation}
\label{eq:Psi-xi}
    \Psi_\xi(x):=(1+h_\xi(x))x.
\end{equation}
Since
\[
    \Psi_\xi(x)
    =
    x+\xi(x)\nu_B(x)
    \qquad\text{for }x\in\partial B,
\]
one has $\Psi_\xi(\partial B)=\partial B^\xi$. Moreover, for
$\|\xi\|_{C^{2,\gamma}(\partial B)}$ sufficiently small,
$\Psi_\xi$ is a $C^{2,\gamma}$-diffeomorphism from $\overline B$ onto
$\overline{B^\xi}$; see, for instance, \cite[Chapter~5]{HP18shape}.

We set
\[
    \widehat u_\xi:=u_\xi\circ\Psi_\xi,
    \qquad
    J_\xi:=\det D\Psi_\xi,
\]
and
\[
    A_\xi
    :=
    J_\xi(D\Psi_\xi)^{-1}(D\Psi_\xi)^{-T}.
\]
We also denote by $J_\xi^\tau$ the tangential Jacobian of
$\Psi_\xi|_{\partial B}$. Changing variables in the weak formulation
of the Robin eigenvalue problem on $B^\xi$, we obtain
\begin{equation}
\label{eq:pullback-weak}
    \int_B
    A_\xi\nabla\widehat u_\xi\cdot\nabla\varphi\,dx
    +
    \beta\int_{\partial B}
    J_\xi^\tau\widehat u_\xi\varphi\,d\Hh^{n-1}
    =
    \lambda_\xi
    \int_BJ_\xi\widehat u_\xi\varphi\,dx
\end{equation}
for every $\varphi\in H^1(B)$. The normalization of $u_\xi$ becomes
\begin{equation}
\label{eq:pullback-normalization}
    \int_BJ_\xi\widehat u_\xi^2\,dx=1.
\end{equation}

We first collect the estimates on the change of variables that will be
used below. Standard estimates for the harmonic extension give
\begin{equation}
\label{eq:harmonic-extension-estimate}
    \|h_\xi\|_{H^1(B)}
    \leq
    C\|\xi\|_{H^{1/2}(\partial B)}
    \leq
    C\|\xi\|_{H^1(\partial B)};
\end{equation}
see, for instance, \cite{DL,AKN2023ACF}. Since
\[
    \Psi_\xi-I=h_\xi x,
    \qquad
    D\Psi_\xi-I=h_\xi I+x\otimes\nabla h_\xi,
\]
we also have
\begin{equation}
\label{eq:Psi-L2}
    \|\Psi_\xi-I\|_{L^2(B)}
    +
    \|D\Psi_\xi-I\|_{L^2(B)}
    \leq
    C\|\xi\|_{H^{1/2}(\partial B)}.
\end{equation}
On the other hand, the corresponding $C^1$ estimates for the harmonic
extension, together with the smooth dependence of the determinant and
the inverse matrix on $D\Psi_\xi$, ensure that
\begin{align}
 \|A_\xi-I\|_{L^\infty(B)}
 +\|J_\xi-1\|_{L^\infty(B)}
 &\leq
 C\|\xi\|_{C^{2,\gamma}(\partial B)},
 \label{eq:coeff-infty}\\
 \|A_\xi-I\|_{L^2(B)}
 +\|J_\xi-1\|_{L^2(B)}
 &\leq
 C\|\xi\|_{H^{1/2}(\partial B)},
 \label{eq:coeff-L2}\\
 \|J_\xi^\tau-1\|_{L^2(\partial B)}
 &\leq
 C\|\xi\|_{H^1(\partial B)}.
 \label{eq:surface-L2}
\end{align}
In particular, for sufficiently small
$\|\xi\|_{C^{2,\gamma}(\partial B)}$, the matrices $A_\xi$ are
uniformly elliptic and the maps $\Psi_\xi$ and $\Psi_\xi^{-1}$ have
uniformly bounded $C^1$ norms.

We will also use uniform regularity of the corresponding
eigenfunctions. After pulling the equation back to the fixed ball,
the coefficients in \eqref{eq:pullback-weak} remain in a bounded
$C^{1,\gamma}$-neighbourhood of the coefficients of the Robin problem
on $B$. Standard elliptic bootstrap, followed by the global Schauder
estimate for oblique derivative problems
\cite[Theorem~6.30]{GT}, therefore gives
\begin{equation}
\label{eq:uniform-regularity}
    \|\widehat u_\xi\|_{C^{2,\gamma}(\overline B)}
    +
    \|u_\xi\|_{C^{2,\gamma}(\overline{B^\xi})}
    \leq C,
\end{equation}
where $C$ is uniform for $\|\xi\|_{C^{2,\gamma}(\partial B)}$ sufficiently small.

For later use, we also record the convergence of the pulled-back
eigenfunctions. Indeed, let $\xi_k\to0$ in
$C^{2,\gamma}(\partial B)$. By \eqref{eq:uniform-regularity} and the
compact embedding of $C^{2,\gamma}(\overline B)$ into
$C^1(\overline B)$, the sequence $\widehat u_{\xi_k}$ is precompact in $C^1(\overline B)$.
The variational characterization of the first eigenvalue
also gives $\lambda_{\xi_k}\to\lambda_B$. Passing to the limit in
\eqref{eq:pullback-weak} and
\eqref{eq:pullback-normalization}, we conclude that every cluster point is
a nonnegative, $L^2$-normalized first Robin eigenfunction of $B$.
By simplicity of the first eigenvalue, every cluster point coincides with $u_B$. Since the
limit is independent of the subsequence, we have proved that
\begin{equation}
\label{eq:C1-convergence}
    \widehat u_\xi\longrightarrow u_B
    \qquad\text{in }C^1(\overline B)
\end{equation}
as $\|\xi\|_{C^{2,\gamma}(\partial B)}\to0$.

We can now prove the comparison estimate on the fixed ball.

\begin{proposition}
\label{prop:pullback-H1}Let \(B=B_R(0)\) be the ball of volume \(V\), and let \(\widehat u_\xi=u_\xi\circ\Psi_\xi\).
There exist $\varepsilon_0=\varepsilon_0(n,\beta,V,\gamma)>0$ and
$C=C(n,\beta,V,\gamma)>0$ such that, if
\[
    \|\xi\|_{C^{2,\gamma}(\partial B)}
    \leq\varepsilon_0,
\]
then
\begin{equation}
\label{eq:pullback-H1}
    \|\widehat u_\xi-u_B\|_{H^1(B)}
    \leq
    C\|\xi\|_{H^1(\partial B)}.
\end{equation}
Consequently, under the assumptions of \autoref{lem:local-coercivity} ,
\begin{equation}
\label{eq:pullback-deficit}
    \|u_\xi\circ\Psi_\xi-u_B\|_{H^1(B)}^2
    \leq
    C\bigl(\lambda_\xi-\lambda_B\bigr).
\end{equation}
\end{proposition}

\begin{proof}
Let
\[
    a_B(v,w)
    :=
    \int_B\nabla v\cdot\nabla w\,dx
    +
    \beta\int_{\partial B}vw\,d\Hh^{n-1}.
\]
We recall that $a_B$ is a scalar product equivalent to the standard one in $H^1(B)$.

We decompose
$\widehat u_\xi$ according to the orthogonal decomposition
\[
    L^2(B)
    =
    \operatorname{span}\{u_B\}
    \oplus
    \{u_B\}^{\perp}.
\]
More precisely, set
\begin{equation}
\label{eq:alpha-v-decomposition}
    \alpha_\xi
    :=
    \int_B\widehat u_\xi u_B\,dx,
    \qquad
    v_\xi
    :=
    \widehat u_\xi-\alpha_\xi u_B.
\end{equation}
Then
\begin{equation}
\label{eq:v-orthogonal}
    \int_Bv_\xi u_B\,dx=0,
    \qquad
    \widehat u_\xi=\alpha_\xi u_B+v_\xi.
\end{equation}
Notice also that $\alpha_\xi>0$, since both $u_B$ and
$\widehat u_\xi$ are positive.

Let $\lambda_{2,\beta}(B)$ denote the second Robin eigenvalue of $B$.
By the variational characterization of $\lambda_{2,\beta}(B)$ and
\eqref{eq:v-orthogonal},
\begin{equation}
\label{eq:second-eigenvalue}
    \lambda_{2,\beta}(B)
    \|v_\xi\|_{L^2(B)}^2
    \leq
    a_B(v_\xi,v_\xi).
\end{equation}
Moreover,
\[
    a_B(u_B,v_\xi)
    =
    \lambda_B\int_Bu_Bv_\xi\,dx
    =
    0,
\]
and therefore
\[
    a_B(\widehat u_\xi,v_\xi)
    =
    a_B(v_\xi,v_\xi).
\]
Taking $\varphi=v_\xi$ in \eqref{eq:pullback-weak} and adding and
subtracting the coefficients of the unperturbed problem, we obtain
\begin{align}
a_B(v_\xi,v_\xi)
={}&
\lambda_\xi\|v_\xi\|_{L^2(B)}^2
+
\lambda_\xi
\int_B(J_\xi-1)\widehat u_\xi v_\xi\,dx
\notag\\
&+
\int_B(I-A_\xi)
\nabla\widehat u_\xi\cdot\nabla v_\xi\,dx
\notag\\
&+
\beta\int_{\partial B}
(1-J_\xi^\tau)
\widehat u_\xi v_\xi\,d\Hh^{n-1}.
\label{eq:spectral-gap-identity}
\end{align}

We estimate the last three terms. By
\eqref{eq:coeff-L2}, \eqref{eq:surface-L2},
\eqref{eq:uniform-regularity}, and the trace inequality,
\begin{align*}
\left|
\lambda_\xi
\int_B(J_\xi-1)\widehat u_\xi v_\xi\,dx
\right|
&\leq
C\|\xi\|_{H^1(\partial B)}
\|v_\xi\|_{H^1(B)},\\
\left|
\int_B(I-A_\xi)
\nabla\widehat u_\xi\cdot\nabla v_\xi\,dx
\right|
&\leq
C\|\xi\|_{H^1(\partial B)}
\|v_\xi\|_{H^1(B)},\\
\left|
\beta\int_{\partial B}
(1-J_\xi^\tau)
\widehat u_\xi v_\xi\,d\Hh^{n-1}
\right|
&\leq
C\|\xi\|_{H^1(\partial B)}
\|v_\xi\|_{H^1(B)}.
\end{align*}
Hence \eqref{eq:spectral-gap-identity} gives
\begin{equation}
\label{eq:v-energy-pre}
    a_B(v_\xi,v_\xi)
    \leq
    \lambda_\xi\|v_\xi\|_{L^2(B)}^2
    +
    C\|\xi\|_{H^1(\partial B)}
    \|v_\xi\|_{H^1(B)}.
\end{equation}

By \eqref{eq:second-eigenvalue},
\[
    \lambda_\xi\|v_\xi\|_{L^2(B)}^2
    \leq
    \frac{\lambda_\xi}{\lambda_{2,\beta}(B)}
    a_B(v_\xi,v_\xi).
\]
Since $\lambda_\xi\to\lambda_B$ and
\[
    \lambda_B<\lambda_{2,\beta}(B),
\]
there exists $\varepsilon_0>0$ such that
\[
    1-\frac{\lambda_\xi}{\lambda_{2,\beta}(B)}
    \geq c>0
\]
whenever
$\|\xi\|_{C^{2,\gamma}(\partial B)}\leq\varepsilon_0$.
Using this in \eqref{eq:v-energy-pre}, together with the equivalence
between $a_B(v,v)$ and $\|v\|_{H^1(B)}^2$, we obtain
\begin{equation}
\label{eq:vxi-bound}
    \|v_\xi\|_{H^1(B)}
    \leq
    C\|\xi\|_{H^1(\partial B)}.
\end{equation}

It remains to estimate the component of $\widehat u_\xi$ in the
direction of $u_B$. From \eqref{eq:pullback-normalization} and the
orthogonal decomposition \eqref{eq:v-orthogonal},
\begin{align*}
1
&=
\int_BJ_\xi\widehat u_\xi^2\,dx\\
&=
\int_B\widehat u_\xi^2\,dx
+
\int_B(J_\xi-1)\widehat u_\xi^2\,dx\\
&=
\alpha_\xi^2
+
\|v_\xi\|_{L^2(B)}^2
+
\int_B(J_\xi-1)\widehat u_\xi^2\,dx.
\end{align*}
Thus, by \eqref{eq:coeff-L2},
\eqref{eq:uniform-regularity}, and
\eqref{eq:vxi-bound},
\[
    |1-\alpha_\xi^2|
    \leq
    \|v_\xi\|_{L^2(B)}^2
    +
    C\|J_\xi-1\|_{L^2(B)}
    \leq
    C\|\xi\|_{H^1(\partial B)}.
\]
Since $\alpha_\xi>0$,
\[
    |1-\alpha_\xi|
    =
    \frac{|1-\alpha_\xi^2|}{1+\alpha_\xi}
    \leq
    C\|\xi\|_{H^1(\partial B)}.
\]
Finally,
\[
    \widehat u_\xi-u_B
    =
    (\alpha_\xi-1)u_B+v_\xi,
\]
and therefore \eqref{eq:vxi-bound} yields
\[
    \|\widehat u_\xi-u_B\|_{H^1(B)}
    \leq
    C\|\xi\|_{H^1(\partial B)}.
\]
This proves \eqref{eq:pullback-H1}. Combining it with
\eqref{eq:local-coercivity} gives
\eqref{eq:pullback-deficit}.
\end{proof}

We now compare the eigenfunctions directly on $B\cap B^\xi$.

\begin{proposition}
    \label{prop:original-coordinates}
There exist $\varepsilon_0=\varepsilon_0(n,\beta,V,\gamma)>0$ and
$C=C(n,\beta,V,\gamma)>0$ such that, if
\[
    \|\xi\|_{C^{2,\gamma}(\partial B)}
    \leq\varepsilon_0,
\]
then
\begin{equation}
\label{eq:intersection-H1}
    \|u_\xi-u_B\|_{H^1(B^\xi\cap B)}
    \leq
    C\|\xi\|_{H^1(\partial B)}.
\end{equation}
Consequently, under the assumptions of
\eqref{eq:local-coercivity},
\begin{equation}
\label{eq:intersection-deficit}
    \|u_\xi-u_B\|_{H^1(B^\xi\cap B)}^2
    \leq
    C\bigl(\lambda_\xi-\lambda_B\bigr).
\end{equation}
\end{proposition}

\begin{proof}
We first compare $u_\xi$ with its pull-back
$\widehat u_\xi=u_\xi\circ\Psi_\xi$ on $B^\xi\cap B$.

Let $x\in B^\xi\cap B$. Since $x\in B$, $\Psi_\xi(x)\in B^\xi$,
and therefore both $u_\xi(x)$ and $u_\xi(\Psi_\xi(x))$ are well
defined. By \eqref{eq:uniform-regularity},
\begin{align}
|u_\xi(x)-\widehat u_\xi(x)|
&=
|u_\xi(x)-u_\xi(\Psi_\xi(x))|
\notag\\
&\leq
\|\nabla u_\xi\|_{L^\infty(B^\xi)}
|\Psi_\xi(x)-x|
\notag\\
&\leq
C|h_\xi(x)|.
\label{eq:pointwise-function-comparison}
\end{align}
Here we used
\[
    \Psi_\xi(x)-x=h_\xi(x)x
\]
and the fact that $|x|\leq R$.

We argue similarly for the gradients. By the chain rule,
\[
    \nabla\widehat u_\xi(x)
    =
    D\Psi_\xi(x)^T
    \nabla u_\xi(\Psi_\xi(x)).
\]
Hence
\begin{align}
|\nabla u_\xi(x)-\nabla\widehat u_\xi(x)|
\leq{}&
|\nabla u_\xi(x)
-\nabla u_\xi(\Psi_\xi(x))|
\notag\\
&+
|(I-D\Psi_\xi(x)^T)
\nabla u_\xi(\Psi_\xi(x))|.
\label{eq:gradient-splitting-pointwise}
\end{align}
Using again \eqref{eq:uniform-regularity},
\[
    |\nabla u_\xi(x)
    -\nabla u_\xi(\Psi_\xi(x))|
    \leq
    C|\Psi_\xi(x)-x|,
\]
while
\[
    |(I-D\Psi_\xi(x)^T)
    \nabla u_\xi(\Psi_\xi(x))|
    \leq
    C|D\Psi_\xi(x)-I|.
\]
Since
\[
    |\Psi_\xi(x)-x|
    \leq C|h_\xi(x)|
\]
and
\[
    |D\Psi_\xi(x)-I|
    \leq
    C\bigl(|h_\xi(x)|+|\nabla h_\xi(x)|\bigr),
\]
we conclude that
\begin{equation}
\label{eq:pointwise-gradient-comparison}
    |\nabla u_\xi(x)-\nabla\widehat u_\xi(x)|
    \leq
    C\bigl(
        |h_\xi(x)|+|\nabla h_\xi(x)|
    \bigr)
\end{equation}
for every $x\in B^\xi\cap B$.

We now integrate the pointwise estimates. From
\eqref{eq:pointwise-function-comparison} and
\eqref{eq:harmonic-extension-estimate},
\[
\begin{aligned}
    \|u_\xi-\widehat u_\xi\|_{L^2(B^\xi\cap B)}
    &\leq
    C\|h_\xi\|_{L^2(B)}\leq
    C\|\xi\|_{H^{1/2}(\partial B)}.
\end{aligned}
\]
Similarly, \eqref{eq:pointwise-gradient-comparison} gives
\[
\begin{aligned}
    \|\nabla u_\xi-\nabla\widehat u_\xi\|_{L^2(B^\xi\cap B)}
    &\leq
    C\|h_\xi\|_{H^1(B)}\leq
    C\|\xi\|_{H^{1/2}(\partial B)}.
\end{aligned}
\]
Therefore
\begin{equation}
\label{eq:u-hatu-comparison}
    \|u_\xi-\widehat u_\xi\|_{H^1(B^\xi\cap B)}
    \leq
    C\|\xi\|_{H^{1/2}(\partial B)}
    \leq
    C\|\xi\|_{H^1(\partial B)}.
\end{equation}

Finally, by the triangle inequality,
\[
\|u_\xi-u_B\|_{H^1(B^\xi\cap B)}
\leq
\|u_\xi-\widehat u_\xi\|_{H^1(B^\xi\cap B)}
+
\|\widehat u_\xi-u_B\|_{H^1(B^\xi\cap B)}\leq
    C\|\xi\|_{H^1(\partial B)},
\]
which proves \eqref{eq:intersection-H1}. The estimate
\eqref{eq:intersection-deficit} follows immediately from
\eqref{eq:local-coercivity}.
\end{proof}

Combining \eqref{eq:pullback-deficit} and
\eqref{eq:intersection-deficit}, we finally obtain
\[
    \|u_\xi\circ\Psi_\xi-u_B\|_{H^1(B)}^2
    +
    \|u_\xi-u_B\|_{H^1(B^\xi\cap B)}^2
    \leq
    C\bigl(\lambda_\xi-\lambda_B\bigr),
\]
which proves \autoref{thm:nearly-spherical}.
\bibliographystyle{plain}
\bibliography{robin_eigenfunction_stability_restructured}

@book{GT,
  author = {Gilbarg, David and Trudinger, Neil S.},
  title = {Elliptic Partial Differential Equations of Second Order},
  series = {Grundlehren der mathematischen Wissenschaften},
  volume = {224},
  publisher = {Springer},
  address = {Berlin--Heidelberg},
  year = {1977},
  doi = {10.1007/978-3-642-96379-7}
}

@article{Faber1923,
  author = {Faber, Georg},
  title = {Beweis, da{\ss} unter allen homogenen Membranen von gleicher Fl{\"a}che und gleicher Spannung die kreisf{\"o}rmige den tiefsten Grundton gibt},
  journal = {Sitzungsberichte der Bayerischen Akademie der Wissenschaften, Mathematisch-physikalische Klasse},
  pages = {169--172},
  year = {1923}
}

@book{HP18shape,
  author = {Henrot, Antoine and Pierre, Michel},
  title = {Shape Variation and Optimization: A Geometrical Analysis},
  series = {EMS Tracts in Mathematics},
  volume = {28},
  publisher = {European Mathematical Society},
  year = {2018},
  doi = {10.4171/178},
  isbn = {9783037191781}
}

@article{DRELICHMAN2008286,
  author = {Drelichman, Irene and Dur{\'a}n, Ricardo G.},
  title = {Improved {Poincar\'e} inequalities with weights},
  journal = {Journal of Mathematical Analysis and Applications},
  volume = {347},
  number = {1},
  pages = {286--293},
  year = {2008},
  doi = {10.1016/j.jmaa.2008.06.005}
}

@article{Krahn1925,
  author = {Krahn, Edgar},
  title = {{\"U}ber eine von {Rayleigh} formulierte Minimaleigenschaft des Kreises},
  journal = {Mathematische Annalen},
  volume = {94},
  pages = {97--100},
  year = {1925}
}

@article{Talenti1976,
  author = {Talenti, Giorgio},
  title = {Elliptic equations and rearrangements},
  journal = {Annali della Scuola Normale Superiore di Pisa, Classe di Scienze (4)},
  volume = {3},
  number = {4},
  pages = {697--718},
  year = {1976}
}

@article{Melas1992,
  author = {Melas, Antonios D.},
  title = {The stability of some eigenvalue estimates},
  journal = {Journal of Differential Geometry},
  volume = {36},
  number = {1},
  pages = {19--33},
  year = {1992}
}

@article{FMP2009,
  author = {Fusco, Nicola and Maggi, Francesco and Pratelli, Aldo},
  title = {Stability estimates for certain {Faber--Krahn}, isocapacitary and {Cheeger} inequalities},
  journal = {Annali della Scuola Normale Superiore di Pisa, Classe di Scienze (5)},
  volume = {8},
  number = {1},
  pages = {51--71},
  year = {2009}
}

@article{BDV2015,
  author = {Brasco, Lorenzo and De Philippis, Guido and Velichkov, Bozhidar},
  title = {{Faber--Krahn} inequalities in sharp quantitative form},
  journal = {Duke Mathematical Journal},
  volume = {164},
  number = {9},
  pages = {1777--1831},
  year = {2015},
  doi = {10.1215/00127094-3120167}
}

@incollection{BrascoDePhilippis2017,
  author = {Brasco, Lorenzo and De Philippis, Guido},
  title = {Spectral inequalities in quantitative form},
  booktitle = {Shape Optimization and Spectral Theory},
  editor = {Henrot, Antoine},
  pages = {201--281},
  publisher = {De Gruyter Open},
  year = {2017}
}

@article{CicaleseLeonardi2012,
  author = {Cicalese, Marco and Leonardi, Gian Paolo},
  title = {A selection principle for the sharp quantitative isoperimetric inequality},
  journal = {Archive for Rational Mechanics and Analysis},
  volume = {206},
  number = {2},
  pages = {617--643},
  year = {2012},
  doi = {10.1007/s00205-012-0544-1}
}

@article{KohlerJobin1978,
  author = {Kohler-Jobin, Marie-Th{\'e}r{\`e}se},
  title = {Une m{\'e}thode de comparaison isop{\'e}rim{\'e}trique de fonctionnelles de domaines de la physique math{\'e}matique. {I}. Une d{\'e}monstration de la conjecture isop{\'e}rim{\'e}trique {$P\lambda^2\ge\pi j_0^4/2$} de {P{\'o}lya} et {Szeg{\H{o}}}},
  journal = {Zeitschrift f{\"u}r angewandte Mathematik und Physik},
  volume = {29},
  number = {5},
  pages = {757--766},
  year = {1978}
}

@article{KohlerJobin1982,
  author = {Kohler-Jobin, Marie-Th{\'e}r{\`e}se},
  title = {Symmetrization with equal {Dirichlet} integrals},
  journal = {SIAM Journal on Mathematical Analysis},
  volume = {13},
  number = {1},
  pages = {153--161},
  year = {1982}
}

@article{AKN2023ACF,
  author = {Allen, Mark and Kriventsov, Dennis and Neumayer, Robin},
  title = {Sharp quantitative {Faber--Krahn} inequalities and the {Alt--Caffarelli--Friedman} monotonicity formula},
  journal = {Ars Inveniendi Analytica},
  year = {2023},
  doi = {10.15781/e7f3-a487},
  note = {Paper No.~1, 49 pp.}
}

@article{AKN2023Linear,
  author = {Allen, Mark and Kriventsov, Dennis and Neumayer, Robin},
  title = {Linear stability implies nonlinear stability for {Faber--Krahn} type inequalities},
  journal = {Interfaces and Free Boundaries},
  volume = {25},
  number = {2},
  pages = {217--324},
  year = {2023}
}

@misc{AKN,
  author = {Allen, Mark and Kriventsov, Dennis and Neumayer, Robin},
  title = {Quantitative Resolvent and Eigenfunction Stability for the {Faber--Krahn} Inequality},
  year = {2025},
  howpublished = {Preprint},
  note = {\href{https://arxiv.org/abs/2504.13053v2}{arXiv:2504.13053v2}}
}

@article{Bossel1986,
  author = {Bossel, Marie-H{\'e}l{\`e}ne},
  title = {Membranes {\'e}lastiquement li{\'e}es: extension du th{\'e}or{\`e}me de {Rayleigh--Faber--Krahn} et de l'in{\'e}galit{\'e} de {Cheeger}},
  journal = {Comptes Rendus de l'Acad{\'e}mie des Sciences de Paris, S{\'e}rie I},
  volume = {302},
  number = {1},
  pages = {47--50},
  year = {1986}
}

@article{Bossel1988,
  author = {Bossel, Marie-H{\'e}l{\`e}ne},
  title = {Membranes {\'e}lastiquement li{\'e}es inhomog{\`e}nes ou sur une surface: une nouvelle extension du th{\'e}or{\`e}me isop{\'e}rim{\'e}trique de {Rayleigh--Faber--Krahn}},
  journal = {Zeitschrift f{\"u}r angewandte Mathematik und Physik},
  volume = {39},
  number = {5},
  pages = {733--742},
  year = {1988},
  doi = {10.1007/BF00948733}
}

@article{Daners2006,
  author = {Daners, Daniel},
  title = {A {Faber--Krahn} inequality for {Robin} problems in any space dimension},
  journal = {Mathematische Annalen},
  volume = {335},
  pages = {767--785},
  year = {2006},
  doi = {10.1007/s00208-006-0753-8}
}

@article{DanersKennedy2007,
  author = {Daners, Daniel and Kennedy, James},
  title = {Uniqueness in the {Faber--Krahn} inequality for {Robin} problems},
  journal = {SIAM Journal on Mathematical Analysis},
  volume = {39},
  pages = {1191--1207},
  year = {2007}
}

@article{BucurDaners2010,
  author = {Bucur, Dorin and Daners, Daniel},
  title = {An alternative approach to the {Faber--Krahn} inequality for {Robin} problems},
  journal = {Calculus of Variations and Partial Differential Equations},
  volume = {37},
  pages = {75--86},
  year = {2010}
}

@article{BucurGiacomini2010,
  author = {Bucur, Dorin and Giacomini, Alessandro},
  title = {A variational approach to the isoperimetric inequality for the {Robin} eigenvalue problem},
  journal = {Archive for Rational Mechanics and Analysis},
  volume = {198},
  pages = {927--961},
  year = {2010},
  doi = {10.1007/s00205-010-0298-6}
}

@article{BucurGiacomini2015,
  author = {Bucur, Dorin and Giacomini, Alessandro},
  title = {{Faber--Krahn} inequalities for the {Robin--Laplacian}: a free discontinuity approach},
  journal = {Archive for Rational Mechanics and Analysis},
  volume = {218},
  pages = {757--824},
  year = {2015},
  doi = {10.1007/s00205-015-0872-z}
}

@book{AmbrosioFuscoPallara2000,
  author = {Ambrosio, Luigi and Fusco, Nicola and Pallara, Diego},
  title = {Functions of Bounded Variation and Free Discontinuity Problems},
  publisher = {Oxford University Press},
  year = {2000}
}

@article{BucurFeroneNitschTrombetti2018,
  author = {Bucur, Dorin and Ferone, Vincenzo and Nitsch, Carlo and Trombetti, Cristina},
  title = {The quantitative {Faber--Krahn} inequality for the {Robin} {Laplacian}},
  journal = {Journal of Differential Equations},
  volume = {264},
  number = {7},
  pages = {4488--4503},
  year = {2018},
  doi = {10.1016/j.jde.2017.12.014},
  note = {Numbered references in the text use \href{https://arxiv.org/abs/1611.06704}{arXiv:1611.06704} (2016)}
}

@article{BGN2021,
  author = {Bucur, Dorin and Giacomini, Alessandro and Nahon, Micka{\"e}l},
  title = {Degenerate free discontinuity problems and spectral inequalities in quantitative form},
  journal = {Archive for Rational Mechanics and Analysis},
  volume = {242},
  number = {1},
  pages = {453--483},
  year = {2021},
  doi = {10.1007/s00205-021-01688-7}
}

@article{BW,
  author = {Bandle, Catherine and Wagner, Alfred},
  title = {Second domain variation for problems with {Robin} boundary conditions},
  journal = {Journal of Optimization Theory and Applications},
  volume = {167},
  pages = {430--463},
  year = {2015},
  doi = {10.1007/s10957-015-0801-1},
  note = {Section and equation numbers in the text refer to \href{https://arxiv.org/abs/1403.2220v2}{arXiv:1403.2220v2}}
}

@article{DL,
  author = {Dambrine, Marc and Lamboley, Jimmy},
  title = {Stability in shape optimization with second variation},
  journal = {Journal of Differential Equations},
  volume = {267},
  number = {5},
  pages = {3009--3045},
  year = {2019},
  doi = {10.1016/j.jde.2019.03.033},
  note = {Preprint: \href{https://arxiv.org/abs/1410.2586v4}{arXiv:1410.2586v4}}
}

@misc{CV,
  author = {Cristoforoni, Emanuele and Villone, Federico},
  title = {Symmetry-breaking and local stability of a two-phase eigenvalue problem in optimal insulation},
  year = {2026},
  howpublished = {Preprint},
  note = {\href{https://arxiv.org/abs/2607.07218v1}{arXiv:2607.07218v1}}
}

@article{AmatoBarbatoMasielloPaoli2024,
  author = {Amato, Vincenzo and Barbato, Rosa and Masiello, Alba Lia and Paoli, Gloria},
  title = {The {Talenti} comparison result in a quantitative form},
  journal = {Annali della Scuola Normale Superiore di Pisa, Classe di Scienze},
  year = {2024},
  doi = {10.2422/2036-2145.202404_010},
  eprint = {2311.18617},
  archiveprefix = {arXiv},
  note = {Published online 20 December 2024, 31 pp.}
}

@article{AmatoBarbatoCitoMasielloPaoli2026,
  author = {Amato, Vincenzo and Barbato, Rosa and Cito, Simone and Masiello, Alba Lia and Paoli, Gloria},
  title = {A quantitative {Talenti}-type comparison result with {Robin} boundary conditions},
  journal = {Annali di Matematica Pura ed Applicata},
  year = {2026},
  doi = {10.1007/s10231-026-01722-z},
  eprint = {2511.11316},
  archiveprefix = {arXiv},
  note = {Published online 1 August 2026}
}

@article{AL26,
  author = {Acampora, Paolo and Lamboley, Jimmy},
  title = {Sharp quantitative {Talenti} inequality in particular cases},
  journal = {Archive for Rational Mechanics and Analysis},
  volume = {250},
  year = {2026},
  doi = {10.1007/s00205-026-02179-3},
  note = {Article No.~24}
}

@article{AlvinoNitschTrombetti2023,
  author = {Alvino, Angelo and Nitsch, Carlo and Trombetti, Cristina},
  title = {A {Talenti} comparison result for solutions to elliptic problems with {Robin} boundary conditions},
  journal = {Communications on Pure and Applied Mathematics},
  volume = {76},
  number = {3},
  pages = {585--603},
  year = {2023},
  doi = {10.1002/cpa.22090}
}
\Addresses
\end{document}